\documentclass[a4paper,11pt]{article}
\usepackage{xcolor}
\usepackage{epsfig}
\usepackage{amsmath}
\usepackage{amssymb}
\usepackage{amsthm}
\usepackage{mathtools}
\usepackage{a4wide}
\usepackage{microtype}
\usepackage{lmodern}
\usepackage[utf8]{inputenc}
\usepackage{float}
\usepackage{bm}
\usepackage{verbatim}
\newcommand{\ds}{\displaystyle }

\newtheorem{thm}{Theorem}[section]
\newtheorem{Proposition}[thm]{Proposition}
\newtheorem{cor}[thm]{Corollary}

\newtheorem{rmk}[thm]{Remark}
\newtheorem{defin}[thm]{Definition}

\newcommand{\beq}{\begin{equation}}
\newcommand{\eeq}{\end{equation}}

\begin{document}

\author{Giuseppe D'Onofrio$^a$, Thomas M. Michelitsch$^b$ \\ Federico Polito$^c$, 
Alejandro P. Riascos$^d$ 
\\[1ex]
\footnotesize{$^a$ Department of Mathematical Sciences, Politecnico di Torino, Italy}  \\
\footnotesize{E-mail: giuseppe.donofrio@polito.it}\\[1ex]
\footnotesize{$^b$ Sorbonne Université, CNRS, Institut Jean Le Rond d’Alembert, F-75005 Paris, France } \\
\footnotesize{E-mail: michel@lmm.jussieu.fr}\\[1ex]
\footnotesize{$^c$ Department of Mathematics ``G.\ Peano'', University of Torino, Italy}  \\
\footnotesize{E-mail: federico.polito@unito.it}\\[1ex]
\footnotesize{$^d$ Departamento de Física, Universidad Nacional de Colombia, Bogot\'a, Colombia}  \\
\footnotesize{E-mail: alperezri@unal.edu.co} 
}

\title{On discrete-time arrival processes and related random motions}
 
\maketitle

\begin{abstract}
\noindent We consider three kinds of  discrete-time arrival processes: transient, intermediate
and recurrent, characterized by a finite, possibly finite and infinite number of events, respectively. 
%
%
In this framework, we study renewal processes which are externally stopped at an independent stopping time which may be defective or non-defective.
For defective stopping time, the resulting arrival process is of intermediate nature.
For non-defective stopping time, the resulting arrival process
is transient, i.e.\! stopped almost surely.
For these processes we obtain finite time and asymptotic properties. 
%
%
Particular attention is devoted to the class of transient renewal processes, that is, renewal processes with defective interarrival times. Among these, we consider two examples: \lq\lq Defective Bernoulli Process'' and the \lq\lq Defective Sibuya Process\rq\rq.
We validate some analytical results using Monte Carlo simulations.
%
%
%
%
%
%
We apply these results
to biased and unbiased random walks on the $d$-dimensional infinite lattice and as a special case on the two-dimensional triangular lattice. We study the spatial propagator of the walker
and its large time asymptotics. 
In particular, we observe the emergence of a superdiffusive (ballistic) behavior in the case of biased walks.
For geometrically distributed stopping times, the propagator converges to a stationary non-equilibrium steady state (NESS), which is universal in the sense that it is independent of the stopped process.
In dimension one, for both light- and heavy-tailed step distributions, the NESS has an integral representation involving alpha-stable distributions.

\bigskip



\end{abstract}
\newpage
\tableofcontents
\section{Introduction}
\label{Intro}

The concept of renewal process provides a simple and powerful model to describe the occurrence in time of random sequences of events. The basic assumption is that the random time intervals between the events (also known as waiting times) are independent and identically distributed (IID).
The simplest models consider waiting times that are exponentially or geometrically distributed in the continuous- and discrete-time setting, respectively. These models are in fact the classical homogeneous Poisson process and the Bernoulli process, both characterized by the Markov property.
Already in the mid-twentieth century, the Markov property was relaxed by considering IID non-exponentially distributed waiting times, generating semi-Markov analogs of the above processes (see e.g.~\cite{cox}).
Given the vastness of the relative literature it would be futile to aim at being exhaustive. Hence, we limit ourselves to recalling some recent related papers: \cite{borovkov,PolitoCahoy2013,godreke,GorenfloMainardi2008,Laskin2003,MiRiaFract-calc2020,MichelitschPolitoRiascos2021}.

On the other hand, the very much related concept of continuous-time random walks (CTRWs), in fact a subclass of continuous-time semi-Markov processes, was introduced in the fundamental work by Montroll and Weiss \cite{MontrollWeiss1968}. Since then, it has been extensively developed in different directions. Among the many papers on the topic we refer to the following:  \cite{KutnerMasoliver2017,MeerschertSikorski2019,MeerschertStraka2014,MetzlerKlafter2000}.

Remarkably, as described in the previously cited papers, models of random motions based on CTRWs proved to be able to explain anomalous diffusion phenomena, for instance, by assuming asymptotically power-law waiting times with fat-tails.

Similarly to the theory in continuous time, semi-Markov generalizations of discrete-time processes and in general semi-Markov chains have been developed as well:
see for instance \cite{BarbuLimnios2008,limnios_2023, PachonPolitoRicciuti2021}, or the semi-Markov generalizations of the discrete-time telegrapher's process, see \cite{SRW2022,GSD-Squirrel-walk2023,mirko2}, to cite a few. 

The works mentioned above consider recurrent renewal processes with IID continuous or discrete waiting times, that is they comprise infinitely many events almost surely. There are, however, classes of arrival processes (processes modeling the occurrence of events) that are transient, that is where only a finite number of events occur almost surely in an infinitely long observation time span. Such processes may sometimes arise as renewal processes which are stopped by an external event independent of the renewal process itself. 

Transient arrival processes are ubiquitous in nature. Real-world examples include the annual migrations of birds or droves which are suddenly stopped or modified by the occurrence of natural catastrophes (collapse or destruction of food resources due to climate change, storms, wildfires, earthquakes, volcanic eruptions, etc.) \cite{Carey2009}. Indeed, the occurrence of catastrophic events may force species to change regular habits (such as the recurrent annual migration for foraging) and have been identified as major driving forces that boost evolution \cite{Grant-etal2016}.
Furthermore, transient processes are also  appropriate models in other fields: in population dynamics \cite{def_galton}, or in finance and gambling where a sequence of independent probabilistic trials, or games, are externally stopped by random or deterministic events (for instance by a stopping rule) \cite{czarna_palm,limnios_2023}.

Although transient arrival processes are widespread phenomena, their appearance in the literature is relatively rare compared to their recurrent counterparts,
for which a well established framework exists \cite{Feller1964,Levy1956,PachonPolitoRicciuti2021}.
Consult \cite[p.19]{BarbuLimnios2008}, for a brief discussion of some features of transient discrete-time renewal processes governed by defective interarrival times.
%
%
In particular, a model for transient Markov processes appears in the paper by Latouche et al. \cite{Latouche-et-al2003} in terms of a random walk model on a two-dimensional infinite graph. The authors derived expressions for the distributions of the absorbing state and time in which the arrival process is stopped. A further model on the distributions of stopping times in compound and renewal processes has been studied by Zacks et al. \cite{Zacks-etal1999}. The authors obtain explicit expressions in the case in which a Poisson process crosses an upper boundary, and introduce an approach for a series of similar problems for compound Poisson processes.

In this work we analyze discrete-time arrival processes which we classify as transient, recurrent
or intermediate, based on having finitely many events with probability one, zero or belonging
to $(0,1)$, respectively. For instance, a discrete-time renewal process with defective IID waiting
times is transient. 
We also consider the class of recurrent discrete-time renewal processes
which are stopped (or absorbed) by an independent random variable, which may be defective. If the stopping time is defective, the resulting process is of
intermediate type, otherwise it is transient.
We investigate the rich dynamics arising from these three types of arrival processes, which
seem to have widespread applications in random walks, anomalous transport and in a cross-disciplinary fields.
%
%
%
%
%
%
In particular, for the cases of defective and non-defective geometrically-distributed absorbing time we conduct a thorough analysis by deriving the discrete probability density function of the resulting processes, their moments together with their asymptotic behavior. A similar analysis is done in the case of stopping times drawn from a discrete version of completely monotone distributions. To this end, we evoke a discrete variant of Bernstein's theorem, 
see for instance \cite{jedidi}.
As a byproduct, we show that if a Bernoulli process is stopped by a further independent, possibly defective, Bernoulli process, i.e. with geometric defective or non-defective stopping time then the resulting dynamics is non-Markovian, and so are the associated time-changed random walks.

Indeed, we make use of such a construction to define some related biased and unbiased random walks in the $d$-dimensional infinite lattice. Properties of recurrent, transient, and intermediate arrival processes are inherited by these random walks via stochastic time-change.
For them, we deduce the basic properties, such as the spatial probability density function of the position of the walker. 
Particular attention is given to the intermediate random walks where the stopping time follows a Bernoulli process with defective waiting times. If this walk is biased we observe the emergence of an asymptotic superdiffusive-ballistic diffusion whose fluctuations reach their maximum when the defect equals $1/2$. If, instead, the walk is unbiased the long-time asymptotics is of a diffusive nature.
We further derive a stationary infinite-time propagator (the NESS) which emerges in a well-scaled continuous space limit as long as the defectivity vanishes. Recall, the concept of NESS frequently arises in the context of stochastic resetting \cite{Barkai-etal-2023,dicre_tel,Eule-et-al2016,Sandev-et-al2022}. Here, the interpretation of the NESS is that it represents the spatial PDF of the rescaled end point of the eventually stopped sample path. This NESS  is universal in the sense that it is of the same kind, irrespectively of the underlying stopped process. These results hold true for both biased and unbiased walks with light- or heavy-tailed step distributions. In particular, in dimension one, the NESS has an integral representation involving alpha-stable distributions.

The paper is organized as follows. 
After a brief recap of some basic properties of discrete-time renewal processes and of defective distributions (Section \ref{preliminaries}),  
we introduce a classification scheme of arrival processes.
As a first example of 
transient arrival processes, we consider defective renewal processes (Section \ref{transient_renewal}).
These are renewal processes in discrete time with IID defective interarrival times. We then consider two special cases, the Defective Bernoulli
process (DBP) (Section \ref{DBP}), and the Defective Sibuya process (DSP) (Section \ref{DSP}), for which we obtain the state probabilities in explicit form and conduct an asymptotic analysis for the event counting process. We also empirically validate these results with Monte Carlo simulations.
In Section \ref{appendix_c}, we summarize some essential features of transient renewal processes.
%
%
Section \ref{pertinent_GFs} is devoted to  the analysis of discrete-time renewal processes stopped at a random time instant. The stopping time is independent of the renewal process and is either defective or non-defective. 
We analyze several cases of  transient and of intermediate types. 
Among the transient ones we consider Bernoulli and Sibuya renewal processes stopped by a non-defective stopping time. For these, we perform sample path simulations to validate the analytical results. The end of Section \ref{pertinent_GFs} is devoted to intermediate arrival processes
which emerge when a recurrent  renewal process is geometrically stopped with a defective stopping time. We then study the case in which the stopped process is a Bernoulli process.
Finally a particular interesting intermediate arrival process emerges when a Bernoulli process is stopped with a defective Sibuya stopping time.
For this arrival process, we obtain (Section \ref{defective_Sibuya_stop}) explicit formulae for the first and second moment of the event counting process. In the limit of a non-defective  Sibuya stopping time this arrival process becomes transient and
the moments diverge for large times with a specific power law behavior.
%
%
%
%
%
%
\\
Random walks governed by all the three kinds of the considered arrival processes are addressed in Section \ref{RWs}. We focus on general biased and unbiased random walks in $d$-dimensional lattices. We explore the classes of random walks whose features are a direct consequence of the classification we introduced in Section \ref{classification}.  
In this regard, the explicit results of Section \ref{pertinent_GFs} allow the application on biased and unbiased random walks in triangular lattices. Finally, the reader might be interested also in the Appendix \ref{type_I_limit}: we describe walks generated by an arbitrary recurrent renewal process with non-defective geometrically distributed stopping time and we look at the behavior of its NESS.

\section{Preliminaries}
\label{preliminaries}
\subsection{Discrete-time renewal processes}
\label{DTRW}

Let us briefly recall the notion of discrete-time renewal (counting) processes
${\cal R}(t) \in \mathbb{N}_0 =\{0,1,\dots\}$ (see e.g.\ \cite{BarbuLimnios2008,ADTRW2021,PachonPolitoRicciuti2021}). First consider the renewal chain process
\beq
\label{renewal_chain}
J_n=\sum_{j=1}^n \Delta t_j, \qquad n \in \mathbb{N}=\{1,2,\dots\}, \hspace{0.5cm} J_0=0,
\eeq
where the increments $\Delta t_j  \in \mathbb{N}$ 
(`interarrival times'  or `waiting times' in the renewal picture) are IID copies of $\Delta t$ characterized by the probability density function (PDF)\footnote{We also use the term `waiting time density' and employ these terms for discrete-and continuous times.}
\beq
\label{first_arrival}
\mathbb{P}[\Delta t=t]= \psi(t) , \hspace{1cm}  t\in \mathbb{N}.
\eeq
Now define
\beq
\label{discrete-time}
{\cal R}(t) = \max(n\geq 0: J_n \leq t)  , \hspace{1cm} {\cal R}(0)=0 , \hspace{1cm} t \in\mathbb{N}_0.
\eeq
%
%
%
%
The process (\ref{discrete-time}) counts the events up to and including time $t$, and it is
the inverse of the discrete subordinator \eqref{renewal_chain}. 
%
%
%
%
%
It is useful to introduce the generating function (GF) of the waiting time density:
\beq
\label{gen_fu}
\mathbb{E} u^{\Delta t} = {\overline \psi}(u) = \sum_{t=1}^{\infty}\psi(t) u^t , \hspace{1cm} |u| \leq 1.
\eeq
Note that ${\overline \psi}(u)\big|_{u=1}=1$ reflects the normalization of the PDF (\ref{first_arrival}). Recall also that ${\overline \psi}(u)\big|_{u=0}= \psi(0)=0$ due to the strict positivity of the waiting times $\Delta t$.

In the following we will denote the GF (or discrete Laplace transform if $u=e^{-s}$) of a discrete density $f(t)$ ($t\in \mathbb{N}_0$) by
${\overline f}(u) = \sum_{t=0}^{\infty} f(t)u^t$ with suitable $|u|$ within the radius of convergence. 
For discrete convolutions of two causal functions (we call a function $f(r)$ causal if $f(r)=0$ for $r<0$), say $a(t),b(t)$, supported on $t\in \mathbb{N}_0$ we will use the equivalence
\beq
\label{discrete-time-convol}
(a \star b)(t)=(b \star a)(t) =: \sum_{r=0}^t a(r)b(t-r) = \frac{1}{t!}\frac{d^t}{du^t}[{\overline a}(u) {\overline b}(u)]\big|_{u=0}.
\eeq
Further, we will denote the discrete convolution powers with $(a \star)^n = (a\star \ldots \star a)(t)$. 
The discrete Heaviside unit step function defined on 
$r\in \mathbb{Z}$ is such that
\beq
\label{discrete-Heaviside}
\Theta(r) =\sum_{k=-\infty}^{r} \delta_{0,k} =  \left\{\begin{array}{cl} 1 ,& r \geq 0 \\ \\
	0 , & r<0 \end{array}\right. \hspace{1cm} (r \in \mathbb{Z}).
\eeq
Note in particular that $\Theta(0)=1$. Finally, we will use $\delta_{r,k}=\delta_{r\,k}$ as equivalent notations for the Kronecker symbol.

Let us now  consider the product
\beq
\label{incomp_PDF}
\psi_{\cal F}(t)=\psi(t){\cal F}(t) , \hspace{1cm} t \in \mathbb{N}
\eeq
in which ${\cal F}(t) \in [0,1]$ for every $t \in \mathbb{N}_0$ 
%
%
%
%
where $\exists \, t_1 \in \mathbb{N} : \, {\cal F}(t_1) >0$
and $\psi_{\cal F}(\infty) = 0$, 
where $\ds \psi_{\cal F}(\infty)$ stands for the infinite time limit $\ds \psi_{\cal F}(\infty)=\lim_{t \to \infty} \ds \psi_{\cal F}(t)$.
%
%
The GF of (\ref{incomp_PDF}) then reads 
\beq
\label{GF-defective}
{\overline \psi}_{\cal F}(u) = \mathbb{E} u^{\Delta t}{\cal F}(\Delta t) = \sum_{t=1}^{\infty} \psi(t) {\cal F}(t) u^t,
\eeq
converging at least for $|u| \leq 1$. 
The special case in which ${\cal F}(t)=1$ for every $t \in \mathbb{N}_0$
gives ${\overline \psi}_{\cal F}(u) = \mathbb{E} \, u^{\Delta t} = {\overline \psi}(u)$, the non defective case.
%
%
Clearly, we have 
\beq
\label{incomplete_normalization}
{\overline \psi}_{\cal F}(u)\bigg|_{u=1} = \mathcal{Q} \in (0,1],
\eeq
where $\psi_{\cal F}(t)$ is not a proper PDF if $\mathcal{Q}<1$. In the latter case we will assume the missing probability mass $1-\mathcal{Q}$ being placed onto $\infty$. 
\begin{defin}\label{defI}
We call random variables with PDF (\ref{incomp_PDF}) and $\mathcal{Q}<1$ `{\it incomplete}' 
or `{\it defective}'.
Accordingly, if $\mathcal{Q}=1$ we refer to `{\it complete}' or `{\it non-defective}' random variables.
\end{defin}
With an abuse of terminology we will also talk of `{\it defective}' and `{\it non-defective}' distributions.
We will consider defective distributions subsequently in more details (Section \ref{pertinent_GFs}).
A straightforward defective case is represented by $ {\cal F}(t)=Q_0^{\, t}$, with $Q_0< 1$ for which the GFs 
${\overline \psi}_{\cal F}(u) = {\overline \psi}(Q_0u)$ and $\left. {\overline \psi}_{\cal F}(u)\right|_{u=1}  =  
{\overline \psi}(Q_0) =\mathcal{Q}$.

Let $R(t)$ be the counting process defined similarly to \eqref{discrete-time} but with a renewal chain $({\cal J}_m)_{m \in \mathbb{N}_0}$ with IID increments having PDF \eqref{incomp_PDF}. Consider the probability
that $R(t)$ exceeds a certain value $n_0$ at time $t$
\beq
\label{probability}
F_{n_0}(t) =
\mathbb{P}[R(t) >n_0] =  \mathbb{E} \Theta(t-{\cal J}_{n_0+1}).
\eeq
Hence, (\ref{probability}) has GF
\beq
\label{GF_Fn0}
{\overline F}_{n_0}(u) = \mathbb{E}\sum_{t={\cal J}_{n_0+1}}^{\infty} u^t  = \frac{[{\overline \psi_\mathcal{F}}(u)]^{n_0+1}}{1-u}.
\eeq
Note that we used the IID property so that $\mathbb{E} u^{{\cal J}_m} = [{\overline \psi_\mathcal{F}}(u)]^{m}$.
This GF in fact contains the large-time limit of $F_{n_0}(t)$. Namely,
\beq
\label{large-timeproba}
F_{n_0}(\infty) = \mathbb{P}[R(\infty) > n_0] = [{\overline \psi_\mathcal{F}}(u)]^{n_0+1}\bigg|_{u=1} = \mathcal{Q}^{n_0+1}
\eeq
i.e.\ $F_{n_0}(\infty)=1$ if $\mathcal{F}(t)=1$ for all $t$ (non-defective). 
Note that the number of arrivals almost surely exceeds any finite number.
On the contrary, if $\psi_\mathcal{F}$ is defective ($\mathcal{Q} < 1$) we have
$\lim_{n_0\to \infty} \mathbb{P}[R(\infty) > n_0] =  0$. 

In Sections \ref{DBP} and  \ref{DSP} we consider some essential features of transient variants of Sibuya and Bernoulli renewal processes, respectively.

Sometimes we evoke in our paper the picture of a renewal process (RP) as a sequence of possibly dependent trials
performed at integer time instants. If the outcome of a trial is a 'success' this corresponds to an event in the renewal picture, whereas if the outcome is a 'fail' the time instant is uneventful.
For our convenience we introduce the boolean random indicator variable which is telling us the outcomes, i.e. 
$Z_t =\{0,1\}$ where $Z_t=1$ (success) 
if $t$ is eventful, and $Z_t=0$ (fail) if $t$ is uneventful. We fix $Z_0=0$ almost surely.
The counting variable (\ref{discrete-time}) thus writes 
\beq
\label{counting_N}
{\cal R}(t)= \sum_{r=1}^t Z_r,
\eeq
indicating the number of events occurred up to and including time $t$. 

\subsection{Classification of arrival processes} \label{classification}

In general, we consider simple counting processes $\mathcal N(t)$, i.e.\ integer-valued stochastic processes with non-decreasing sample paths and such that, with probability $1$, there is no time at which two or more events occur simultaneously. 
In analogy with \cite{Latouche-et-al2003} we refer to $\mathcal{N}(t)$ as an arrival process (AP).
In the present paper we consider distinct classes of APs:
\begin{description}
\item {\bf Transient arrival processes} -- ({\bf Type I AP}): \, $\mathbb{P}[{\cal N}(\infty) < \infty]=1$, corresponding to an almost surely finite number of events with waiting times that can be infinite with strictly positive probability. This class contains discrete-time renewal processes with defective inter-arrival distribution (\lq defective RP' for simplicity), introduced  above.
\item {\bf Recurrent arrival processes} -- ({\bf Type II AP}):\,
 $\mathbb{P}[{\cal N}(\infty) = \infty] = 1$, corresponding to an infinite number of events with almost surely finite waiting times.  This class contains discrete-time renewal processes with non-defective inter-arrival distribution.
\item {\bf Intermediate arrival processes (IAPs)}: \, $\mathbb{P}[{\cal N}(\infty) < \infty] = {\cal Q}$, \, ${\cal Q}\in (0,1)$. 
IAPs can be viewed as a mixture of type I and type II processes.

\end{description}
 %
 %
 %
 %
%
%
The above classification includes the classes of transient and recurrent renewal processes mentioned in reference \cite{BarbuLimnios2008}.
%
%

Let
$P(n_0,t)= \mathbb{P}[{\cal N}(t) > n_0]$
and define the limiting quantity
\beq
\label{Lambda}
\Lambda = \lim_{n_0 \to \infty}  \lim_{t \to \infty} P(n_0,t) = \lim_{n_0 \to \infty} P(n_0,\infty).
\eeq
%
%
Clearly, $\Lambda =\mathbb{P}[{\cal N}(\infty)=\infty]$, that is the probability that the AP ${\cal N}(t)$ never stops.
Note that, for type I APs, $\Lambda = 0$, while for type II APs, $\Lambda = 1$, and finally, $\Lambda = 1 - \mathcal{Q} \in (0,1)$ for IAPs.
%
%
%
%
%
%
%
%
%
%
%
%
%
%
%

\section{Transient renewal processes}
\label{transient_renewal}
%
%
Here we consider a class of type I APs, namely transient renewal processes with defective IID interarrival times, i.e., they have a defective PDF $\psi_I(t)$, that is $\sum_{t=1}^{\infty}\psi_{I}(t) = \mathcal{Q} < 1$. We describe hereafter two special cases.
%
%

\subsection{Defective Bernoulli process (DBP)} 
\label{DBP}
We introduce a type I renewal process which we refer to as `defective Bernoulli process' (DBP).
Here we denote with $N_I(t)$ the DBP counting variable and
define the DBP by the waiting time GF
\beq
\label{DFB-GF}
{\overline \psi}_{I}(u) = \mathcal{Q}\frac{pu}{1-qu} ,\hspace{1cm} p+q=1 , \hspace{1cm} \mathcal{Q}\in (0,1],
\eeq
of the PDF $\psi_{I}(t)=\mathcal{Q}pq^{t-1}$, $t\in \mathbb{N}$. 
This gives straightforwardly the survival probability
\beq
\label{DBP-survival}
\Phi_{I}^{(0)}(t)= \mathcal{P}+\mathcal{Q}q^t
\eeq
with $\mathcal{P}=1-\mathcal{Q}$ and  $\Phi_{I}^{(0)}(0)=1$. The probability
$\Phi_{I}^{(0)}(t) \to \mathcal{P}$ as $t\to \infty $ and 
it has the GF 
\beq
\label{DFP-suvivalGF}
{\overline \Phi}_{I}^{(0)}(u) = \frac{\mathcal{P}}{1-u}+ \frac{\mathcal{Q}}{1-qu}.
\eeq
For $\mathcal{Q}=1$ all these relations turn into those of the Bernoulli process. 
Now, consider the state probabilities $\mathbb{P}[N_{I}(t)=n]=\Phi_{I}^{(n)}(t)$, i.e.\ the probability that $N_I$ counts $n$ events up to and including time $t$. Its GF reads (cf.~\eqref{GF_state_II})
\begin{align}
    \label{DBP-state-prob}
    {\overline \Phi}_{I}^{(n)}(u) & = \sum_{t=0}^{\infty}u^t \Phi_{I}^{(n)}(t) = {\overline \Phi}_{I}^{(0)}(u) [{\overline \psi}_{I}(u)]^n , \qquad  n,t \in \mathbb{N}_0, \qquad |u|<1, \\
    & = \left(\frac{\mathcal{P}}{1-u}+ \frac{\mathcal{Q}}{1-qu}\right)\frac{(\mathcal{Q} pu)^n}{(1-qu)^n}. \notag
\end{align}
The sum over $n$ of \eqref{DBP-state-prob} is $(1-u)^{-1}$ and this corresponds to the normalization $\sum_{n=0}^{\infty}\Phi_{I}^{(n)}(t)=1$.
Further, we see that 
\beq
\label{large-t-state}
\lim_{t \to \infty}{ \Phi}_{I}^{(n)}(t) = \lim_{u\to 1} (1-u){\overline \Phi}_{I}^{(n)}(u)= \mathcal{P} \mathcal{Q}^n 
\eeq
i.e.\ asymptotically the probabilities that a number $n$ of events occur tends geometrically to zero as $n \to \infty$.
This is different from the non-defective case in which they are zero regardless of the value of $n$.
By inverting (\ref{DBP-state-prob}) we obtain for $n > 0$
\beq
\label{DBP-state-probs}
\ds \Phi_{I}^{(n)}(t)  = \ds \mathcal{Q}^n p^n\left(\mathcal{Q}\binom{t}{n}q^{t-n} + 
\frac{\mathcal{P}}{q^n}\sum_{r=n}^t \binom{r-1}{n-1} q^{r} \right), 
\eeq
while for $n=0$ we have the survival probability (\ref{DBP-survival}).
We observe that $\Phi_{I}^{(n)}(t)=0$ for $t<n$,  thus the initial condition $\Phi_{I}^{(n)}(0)=\delta_{n0}$ is fulfilled.
In the type II limit $\mathcal{Q} \to 1-$, (\ref{DBP-state-probs}) reduces to the
binomial distribution  of  Bernoulli.

\begin{figure}
\centerline{\includegraphics[width=0.9\textwidth]{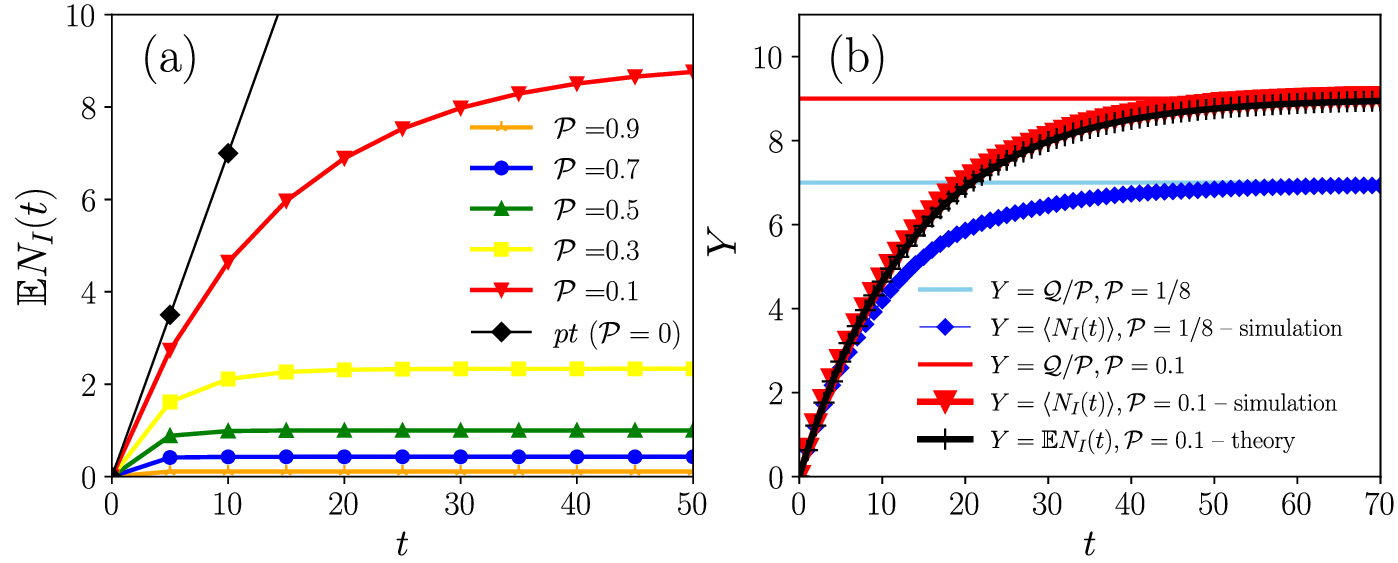}}
\vspace{-4mm}
\caption{
(a): Plot of  $\mathbb{E} N_{I}(t)$ in (\ref{expnum}) for different values of $\mathcal{P}$ and fixed $p=0.7$. For small $t$ a linear behavior of order $\mathcal{Q} pt$ is observed, and for large $t$ the asymptotic value $\mathcal{Q}/\mathcal{P}$ is approached. 
(b): Sample average $\langle N_I(t)\rangle$ of $1000$ realizations  obtained by Monte Carlo simulations compared to (\ref{expnum}) for fixed $p=0.7$.
}
\label{arrivals_def_ber}
\end{figure}
We are now interested in the expected number of DBP arrivals which, by \eqref{GF_NII} and \eqref{DFB-GF}, has the GF
\begin{align}
\label{exp_num-DBP-arrivals}
 \overline{\mathbb{E}  {N}_{I}}(u)   =  \frac{\partial }{\partial v}\bar {\cal P}_{I}(v,u)\Big|_{v=1} 
  =  \frac{1}{1-u}\frac{{\overline \psi}_{I}(u)}{1-{\overline \psi}_{I}(u)} = \frac{\mathcal{Q}pu}{(1-u)[1-u(1-p\mathcal{P})]}    
\end{align}
with $ \bar {\cal P}_{I}(v,u)  = \sum_{t=0}^{\infty}u^t \mathbb{E}  v^{{N}_{I}(t)}$. 
We can invert this GF getting the formula
\beq
\label{expnum}
 \mathbb{E}   N_{I}(t)  = \mathcal{Q} p \sum_{r=1}^t (1-p\mathcal{P})^{r-1} =  \mathcal{Q}p \sum_{r=0}^{t-1} (1-p\mathcal{P})^{r}= 
 \begin{cases}
    pt, & \mathcal{Q} = 1, \\
    \frac{\mathcal{Q}}{\mathcal{P}}\left(1-[1-p\mathcal{P}]^t\right), & \mathcal{Q} \in (0,1),
 \end{cases}
 \eeq
where, for $\mathcal{Q} \in (0,1)$, the number of arrivals approaches geometrically the finite asymptotic value
$\mathbb{E}   N_{I}(\infty) =
\mathcal{Q}/\mathcal{P}, 
$
which diverges in the type II limit $\mathcal{P}\to 0+$.
We notice that in the type II limit $\mathcal{P}\to 0+$ the expected number of arrivals increases linearly in time as for the  Bernoulli process. Namely,
\beq
\label{limit-to-type-II}
\lim_{\mathcal{P}\to 0+}  \mathbb{E}  N_{I}(t)  = \lim_{\mathcal{P}\to 0+} \frac{1}{\mathcal{P}}\left(1-[1- p\mathcal{P}t+O(\mathcal{P}^2)]\right) =  pt \, .
\eeq
%
%
In Fig.~\ref{arrivals_def_ber} we show $\mathbb{E} N_{I}(t)$ for some values of $\mathcal{P}$ 
while keeping the Bernoulli success probability fixed at $p=0.7$. Increasing the defectiveness of the process (i.e., increasing $\mathcal{P}$) results in a smaller expected number of events at any given time $t$.
The non-defective case ($\mathcal{P}=0$, in black in Fig.~\ref{arrivals_def_ber}) recovers the linear growth of the Bernoulli process whose state probabilities are given by the binomial distribution, see (\ref{DBP-state-probs}). 
For large time, the asymptotic value $\mathcal{Q}/\mathcal{P}$ is approached, independently of $p$.
Fig.~\ref{arrivals_def_ber}-(b), shows the sample average over $1000$ realizations of the counting variable $N_I(t)$ as a function of time. We notice the simulations match the analytical expressions. See Appendix \ref{simulation_method} for a brief sketch of the simulation method.
%
%
%
%
%
\\
In Fig.~\ref{state_probas_def_ber} we depict the state probabilities $\Phi^{(n)}(t)$ for the DBP evaluated at the specific time instant $t=6$ for several values of $\mathcal{P}$.
\begin{figure}
\centerline{\includegraphics[width=0.7\textwidth]{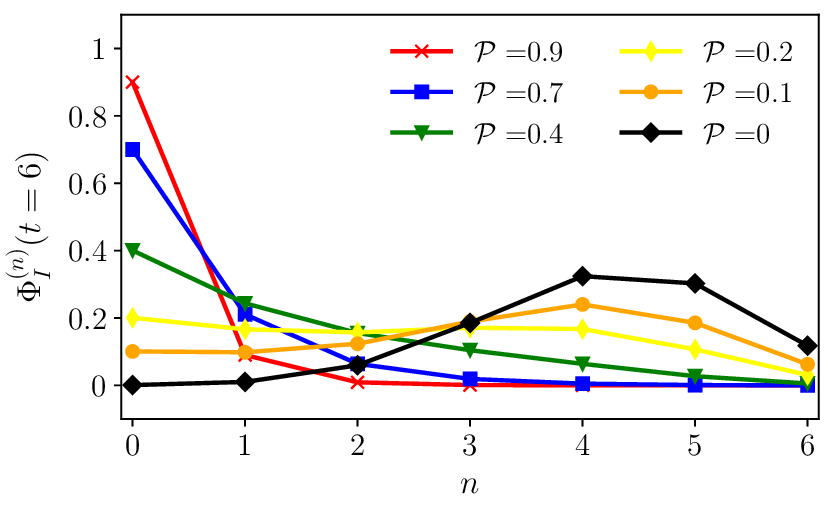}}
\vspace{-4mm}
\caption{State probabilities of DBP (Equations (\ref{DBP-state-probs}) and (\ref{DBP-survival})) at $t=6$ as a function of the number of arrivals $n$ for several values of $\mathcal{P}$ and fixed $p=0.7$. The binomial distribution of the non-defective case ${\cal P}=0$ is drawn in black color.}
\label{state_probas_def_ber}
\end{figure}

\subsection{Defective Sibuya Process (DSP)}
\label{DSP}
As a further prototypical example for a type I renewal process, we consider now the defective Sibuya process (DSP) with waiting time GF
\beq
\label{newSibuya}
{\overline \psi}_{I}(u)= \mathcal{Q}[1-(1-u)^{\mu}] , \hspace{1cm} \mathcal{Q} \in (0,1], \, \mu \in (0,1).
\eeq
The DSP waiting time density then is
\beq
\label{DSP_den}
\psi_{I}(t) = \mathcal{Q}(-1)^{t-1} \binom{\mu}{t} , \hspace{1cm} t \in \mathbb{N},
\eeq
We define the DSP as the renewal process with IID defective interarrival times which follow the PDF (\ref{DSP_den}).
The large-time asymptotics of (\ref{DSP_den}) is
\beq
\label{long-times}
\psi_{I}(t) \sim \mathcal{Q} \frac{\mu}{\Gamma(1-\mu)}t^{-\mu-1},
\eeq
and the DSP survival probability has the GF
\beq
\label{survivalGF}
{\overline \Phi}^{(0)}_{I}(u) = \frac{1-{\overline \psi}_{I}(u)}{1-u} = \frac{\mathcal{P}}{1-u} +\mathcal{Q}(1-u)^{\mu-1}
\eeq
where $\mathcal{P}=1-\mathcal{Q}$ can also be interpreted as the probability that no event has occurred in any finite time. 
Then, by conditioning, we can derive the state probabilities of the DSP as follows: 
\beq
\label{DSP-state-prob}
\Phi^{(n)}_{I}(t) = \mathbb{P}[N_{I}(t)=n] = \left(\Phi^{(0)}_{I} \star (\psi_{I} \star)^n\right)(t) =  
\frac{1}{t!}\frac{d^t}{du^t}\left({\overline \Phi}^{(0)}_{I}(u) [{\overline \psi}_{I}(u)]^n\right)\Big|_{u=0}.
\eeq
In particular, we have for $n=0$,
\beq
\label{large-time-survival}
\Phi_{I}^{(0)}(t) = \mathcal{P} + \mathcal{Q}(-1)^{t} \binom{\mu-1}{t},
\qquad t \in \mathbb{N}_0.
\eeq
We are now interested in the asymptotic expected number of events $\mathbb{E} {N}_{I}(t)$ of a DSP. 
For this purpose recall the notation
%
%
%
%
\beq {\cal P}_{I}(v,t) = \mathbb{E} v^{{N}_{I}(t)}
= \sum_{n=0}^tv^n\Phi_{I}^{(n)}(t).
\eeq
In order to get the large-time behavior, consider the GF of $\mathbb{E}  {N}_{I}(t)$ that, using (\ref{DSP-state-prob}), reads
\beq
\label{GF-expect}
\begin{array}{clr}
\ds  \overline{\mathbb{E}  { N}_{I}}(u)  &
= \ds  \sum_{t=0}^{\infty} u^t  \mathbb{E}  {N}_{I}(t)
\ds = {\overline \Phi}^{(0)}_{I}(u)  \sum_{n=0}^{\infty} n v^{n-1} [{\overline \psi}_{I}(u)]^n \Big|_{v=1} & \\ \\
& = \ds \frac{\partial }{\partial v} 
\frac{1-{\overline \psi}_{I}(u)}{1-u} \sum_{n=0}^{\infty} v^n [{\overline \psi}_{I}(u)]^n \Big|_{v=1} = \frac{\partial }{\partial v} \frac{1-{\overline \psi}_{I}(u)}{(1-u)[1-v{\overline \psi}_{I}(u)]}\Big|_{v=1} & \\ \\
 & = \ds  \frac{{\overline \psi}_{I}(u)}{(1-u)[1-{\overline \psi}_{I}(u)]} = \frac{\mathcal{Q}[1-(1-u)^{\mu}]}{(1-u)[\mathcal{P}+\mathcal{Q}(1-u)^{\mu}]} & 
 \end{array}
\eeq
which can be further written as 
\beq
\label{asymp}
\begin{array}{clr}
  \overline{\mathbb{E}  { N}_{I}}(u)   & = 
\ds \frac{\mathcal{Q}}{\mathcal{P}(1-u)}\frac{1-(1-u)^{\mu}}{1+\frac{\mathcal{Q}}{\mathcal{P}}(1-u)^{\mu}}
\\ \\ 
& = \ds \frac{\mathcal{Q}}{\mathcal{P}(1-u)}\left(1- \frac{(1-u)^{\mu}}{\mathcal{P}[1+\frac{\mathcal{Q}}{\mathcal{P}}(1-u)^{\mu}]}\right) = \frac{\mathcal{Q}}{\mathcal{P}(1-u)} -\frac{\mathcal{Q}}{\mathcal{P}^2}{\overline {\cal E}}_{\mu}(u).&
\end{array}
\eeq
Remarkably, the first line of \eqref{asymp} contains the waiting time GF of the so-called fractional Bernoulli process of type A, first introduced and analyzed in \cite{PachonPolitoRicciuti2021} (see equation (75) in that paper):
\beq
\label{FBA-GF}
{\overline \psi}_{\text{fBa}}(u) = \frac{1-(1-u)^{\mu}}{1+\frac{\mathcal{Q}}{\mathcal{P}}(1-u)^{\mu}}.
\eeq
Its inversion, the waiting time PDF $\psi_{\text{fBa}}(t)$ is a discrete approximation of the Mittag--Leffler distribution. In the second line of expression (\ref{asymp}) the quantity
\beq
\label{discrete-ML-function}
{\overline {\cal E}}_{\mu}(u)=\frac{(1-u)^{\mu-1}}{1+\frac{\mathcal{Q}}{\mathcal{P}}(1-u)^{\mu}}
\eeq
denotes the GF of ${\cal E}_{\mu}(t)$ which is a discrete approximation of the
Mittag--Leffler function (see \cite{MichelitschPolitoRiascos2021,PachonPolitoRicciuti2021} for explicit formulas).
Then by invoking Tauberian arguments we can see that the limit $u\to 1-$ in the GF \eqref{discrete-ML-function} yields the large time asymptotics of $\mathbb{E}N_I(t)$.
Indeed, for large values of $t$, the expected number of events approaches a constant value with a power-law rate, namely
\beq
\label{inverse_asyp}
\begin{array}{clr} 
\ds  \mathbb{E}  {N}_{I}(t)  & = \ds  \frac{\mathcal{Q}}{\mathcal{P}} -\frac{\mathcal{Q}}{\mathcal{P}^2}{\cal E}_{\mu}(t)   & \\ \\
& \sim \ds  \frac{\mathcal{Q}}{\mathcal{P}} -\frac{\mathcal{Q}}{\mathcal{P}^2} \frac{t^{-\mu}}{{\Gamma(1-\mu)}} & \\ \\
& \ds \underset{t \to \infty}{\xrightarrow{\qquad}} \frac{\mathcal{Q}}{\mathcal{P}}, &
\end{array}
\eeq
where
we used that
${\cal E}_{\mu}(t)\sim \frac{t^{-\mu}}{\Gamma(1-\mu)}$ for $t\to \infty$ (as ${\bar{\cal E}}_{\mu}(u) \sim (1-u)^{\mu-1}$ for $u \to 1-$).
One can see in (\ref{GF-expect}) that for $\mathcal{Q}=1$ we have $\overline{\mathbb{E} {N}_{I}}(u) \sim (1-u)^{-\mu-1}$ as $u \to 1-$. This turns into the well-known diverging power-law behavior of the expected number of events of a type II Sibuya process, that is
\beq
\label{standard_Sib}
 \mathbb{E}  N_{I}(t) \sim 
\frac{t^{\mu}}{\Gamma(1+\mu)} \underset{t \to \infty}{\xrightarrow{\qquad}} \infty.
\eeq
%
%
%
Since we do not have an  explicit expression for $\mathbb{E}  {N}_{I}(t)$, we provide in Fig. \ref{DSP_counting_var_simu} a Monte Carlo simulation based on $1000$ realizations of $N_I(t)$. 
We consider here
a Sibuya parameter $\mu=0.85$ and $\mathcal{P}=2/3$. With this choice of parameters we have a relatively rapid approach of power-law type to the asymptotic value $\mathcal{Q}/\mathcal{P}$ predicted in (\ref{inverse_asyp}).
%
%
%
\begin{figure}
\centerline{\includegraphics[width=0.7\textwidth]{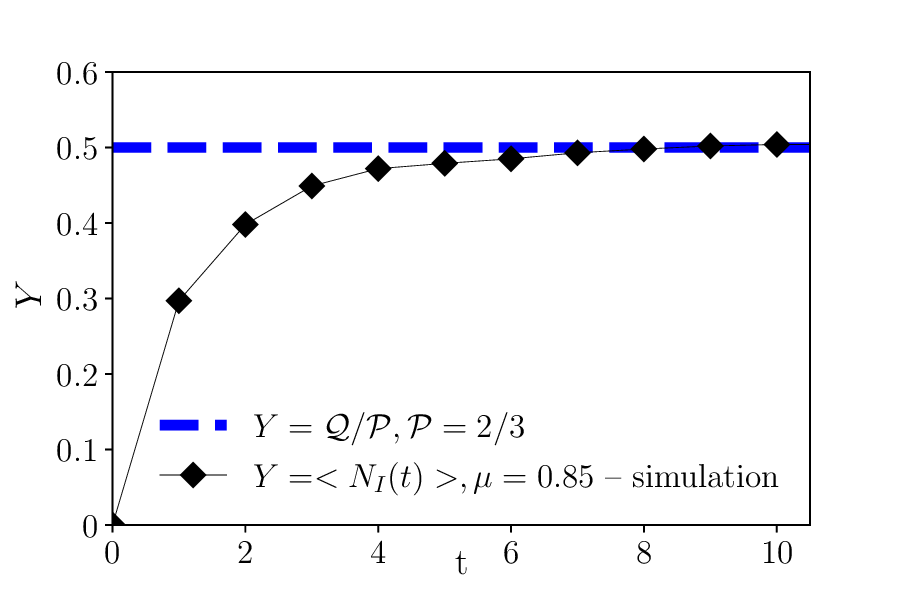}}
\vspace{-3mm}
\caption{ 
Sample path average over $1000$ realizations of $N_i(t)$ for the DSP with Sibuya parameter $\mu=0.85$  and 
 $\mathcal{P}=2/3$. For large times the
predicted asymptotics $\mathcal{Q}/\mathcal{P}$ of (\ref{inverse_asyp}) is approached.
}
\label{DSP_counting_var_simu}
\end{figure}

\subsection{General comments on defective renewal processes}
\label{appendix_c}

The fact that $\mathbb{E}N_I(\infty) = \mathcal{Q}/\mathcal{P}$
(see \eqref{expnum} for the DBP and (\ref{inverse_asyp}) for the DSP) holds in general for defective RPs. Indeed, consider a type I process with waiting time PDF $\psi_I(t)$ such that  
${\overline \psi}_I(u)=\mathcal{Q}{\overline \psi}_{II}(u)$ 
where ${\psi}_{II}(t)$ is non-defective.
Then, the expected number of arrivals of the type I RP has the GF
\beq
\label{GF_expected}
 \overline{\mathbb{E}  {N}_I}(u)  = 
\frac{\mathcal{Q}}{\mathcal{P}(1-u)} \frac{\mathcal{P}{\overline \psi}_{II}(u)}{1-\mathcal{Q}{\overline \psi}_{II}(u)}
= \frac{\mathcal{Q}}{(1-u)} \frac{{\overline \psi}_{II}(u)}{1-\mathcal{Q}{\overline \psi}_{II}(u)} 
\eeq 
which, by Tauberian arguments, gives $\mathbb{E}{N}_I(\infty) = \mathcal{Q}/\mathcal{P}$.
This result is corroborated by the following interpretation: at each renewal time a Bernoulli trial with success probability $\mathcal{P}$ is performed. If a success is obtained the related waiting time is infinitely long, otherwise it is sampled from $\psi_{II}$. Hence, in the limit $t \to \infty$ the random number of successes is geometric of parameter $\mathcal{P}$ regardless of the actual form of $\psi_{II}$. More formally.
this can be inferred from \eqref{large-t-state} which holds true for any renewal process with defective PDF (\ref{tablet}) such that ${\overline \psi}_I(u)\big|_{u=1}= \mathcal{Q} < 1$. Therefore, 
we have the general result
\beq
\label{infinite_number_state}
\Phi_I^{(n)}(\infty) = (1-u)\Phi_I^{(0)} [{\overline \psi}_I(u)]^n \Big|_{u=1} = 
[1-{\overline \psi}_I(u)] [{\overline \psi}_I(u)]^n \Big|_{u=1}
= \mathcal{P} \mathcal{Q}^n, \qquad n \in \mathbb{N}_0.
\eeq
This means that, for $t\to \infty$ the state probabilities are a geometric distributed.
From this relation, one immediately see the type I nature of the renewal processes $N_I(t)$ (see (\ref{Lambda}))
\beq
\label{type_I_confirm}
\Lambda = \lim_{n_0\to \infty} \mathbb{P}[N_I(\infty) > n_0] = \lim_{n_0\to \infty} \sum_{n=n_0+1}^{\infty}\Phi_I^{(n)}(\infty) =  \lim_{n_0\to \infty} \mathcal{Q}^{n_0+1} =0.
\eeq
Note that the process $N_I(t)$ and the process $M(t)$ considered in Section \ref{geometric_stop},
represent different kinds of type I APs: $N_I(t)$ is a renewal process with IID defective interarrival times, whereas $M(t)$ is an externally stopped arrival process and, as we have already shown, it is not a renewal process. In the case in which $M(t)$ has geometric non-defective absorbing time, the large time asymptotics of $M(t)$ is governed by an auxiliary renewal process ${\cal R}_q(t)$ which has defective IID interarrival times (see the GF (\ref{GFPm}) and the remarks thereafter).
\\
Finally, we observe in \eqref{GF_expected} (when $\mathcal{Q} < 1$) that the term $ \frac{\mathcal{P}{\overline \psi}_{II}(u)}{1-\mathcal{Q}{\overline \psi}_{II}(u)}$ is the GF of the waiting time PDF of the time-changed type II RP
${\cal N}({N}_{II}(t))$ with a  Bernoulli RP with success probability  $\mathcal{P}$, and which is independent of $N_{II}$ (see e.g. \cite{ADTRW2021}).
%
%

%
%
%
%
%
\section{Externally stopped arrival processes}
\label{pertinent_GFs}
\subsection{Non-defective $\bm{\Delta T}$}

Consider now a non-defective RP (that is a renewal type II AP) $N_{II}(t)$  with $N_{II}(0)=0$
%
%
which is externally stopped at the random time instant $\Delta T \in \mathbb{N}$ independent of 
$N_{II}(t)$. We assume, for the moment, that the stopping time $\Delta T$ 
is non-defective (recall Definition \ref{defI}), i.e. is drawn from a non-defective PDF $\psi_S(t)$. Hereafter we will denote by $\ds \Phi_S(t)$ the survival probability  of $\Delta t$. 
\begin{rmk}\label{quattrouno}
 In general one can consider  a non-defective RP \(N_{II}(t)\) which is externally stopped at the random time instant, independent of $N_{II}(t)$, defined as
$
\Delta T = \inf \{ r \in \mathbb{N} : N_S(r) \geq 1 \},
$
i.e., at the first arrival time of a further type II RP, say \(N_S(t)\), independent of \(N_{II}(t)\). 
This setting allows one to describe more general scenarios in which, for instance, the counting process is repeatedly interrupted and subsequently restarted. That is, the subsequent events of $N_S(t)$ are also taken into consideration.
\end{rmk}
%
%
%
%
%
%
%
%
We denote with $J_n^{II}$ the renewal times of $N_{II}(t)$ (recall (\ref{renewal_chain})). 
Consider now the AP  $M(t) \in \mathbb{N}_0$, $t\in \mathbb{N}_0$, such that 
\beq
\label{conditional-process}
M(t) = \Theta(\Delta T -1-t)N_{II}(t) + \Theta(t-\Delta T )N_{II}(\Delta T ) ,\hspace{1cm} M(0)=0, \hspace{1cm} \Delta T \in\mathbb{N}.
\eeq
The main goal of this section is to explore the properties of $M(t)$.
We observe that $M(t)$ behaves as $N_{II}(t)$ for $t<\Delta T$ and it is randomly stopped at $\Delta T \in \mathbb{N}$, the `stopping time' (we will also use the term `absorbing time'). After the stopping, the process keeps the 
value $M(\Delta T) = N_{II}(\Delta T)$ forever. 
The contribution $\Theta(\Delta T -1-t)N_{II}(t)$ can be seen as a resetting term: it behaves
as $N_{II}(t)$ for $t<\Delta T$ and it is set to zero at the absorbing time instant $t= \Delta T$.
For what follows, it is convenient to introduce the indicator function
\begin{align}
\label{Theta-a-b}
\Theta(a,t,b) &=\Theta(t-a)-\Theta(t-b) \notag \\
&= \sum_{r=a}^{b-1} \delta_{t,r} =\left\{\begin{array}{l} 1, \hspace{1cm} {\rm for} \hspace{0.5cm} a \leq t \leq b-1  \\ \\ 
0, \hspace{1cm} {\rm otherwise} \end{array}\right. (a,b,t \in \mathbb{N}_0 ,\hspace{0.25cm} a<b).
\end{align}
In particular, we have $\Theta(J_n^{II},t,J_{n+1}^{II})=1$ when $N_{II}(t)=n$ and zero otherwise (consult also \cite{SRW2022} for some details).
Then, we can represent $N_{II}(t)$ as
\beq
\label{NII-RP-theta}
N_{II}(t)= \sum_{\ell=1}^{\infty}\Theta(t-J^{II}_{\ell}),
\eeq
where the summation starts at $\ell=1$ to ensure that $N_{II}(0)=0$. 
\\[3ex]
%
%
%
The following results (Propositions \ref{PropoA}-\ref{propprop2}) can be obtained by classical methods related to random sums (e.g.\ Theorem 25 in \cite{Grimmet_Stirzacher}).
%
\begin{Proposition}
\label{PropoA}
Let ${\cal P}_{II}(v,t)=\mathbb{E} v^{N_{II}(t)}$ be the state polynomial of $N_{II}(t)$ . Then its GF ${\overline {\cal P}}_{II}(v,u)$ is
\beq 
\label{state_poly}
\begin{array}{clr}
\ds {\overline {\cal P}}_{II}(v,u) & = \ds 
\frac{1-{\overline \psi}_{II}(u)}{1-u} \frac{1}{1-v{\overline \psi}_{II}(u)} , & \ds  |v| \leq 1, \hspace{0.5cm} |u| <1.
\end{array}
\eeq
\end{Proposition}
%
%
%
%
%
%
%
%
%
%
We will further need the state probabilities
$\mathbb{P}[N_{II}(t)=n] = \Phi_{II}^{(n)}(t) = \mathbb{E} \Theta(J_n^{II},t,J_{n+1}^{II})$ 
having the GFs
%
%
%
\beq
\label{GF_state_II}
\begin{array}{clr}
\ds {\overline \Phi}_{II}^{(n)}(u) & = \ds \mathbb{E} \sum_{t=J_n^{II}}^{J_{n+1}^{II}-1} u^t =
\mathbb{E} u^{J_n^{II}} \sum_{r=0}^{\Delta t_{n+1}-1} u^r  & \\ \\ & = \ds \frac{1- \mathbb{E} u^{\Delta t_{n+1}}}{1-u} \prod_{j=1}^n \mathbb{E} u^{\Delta t_j} = [{\overline \psi}_{II}(u)]^n\frac{1-{\overline \psi}_{II}(u)}{1-u}  & \hspace{0.5cm} |u| < 1, \hspace{0.5cm} n \in \mathbb{N}_0.
\end{array}
\eeq 
%
%
%
%
%
%
%
%
%
%
%
%
%
%
%
%
%
%
%
%
%
\begin{Proposition}\label{propprop}
The GF of the expected number of arrivals $\mathbb{E} N_{II}(t)$ reads 
\beq
\label{GF_NII}
\overline{\mathbb{E}  N_{II}}(u)  = 
\frac{\partial}{\partial v}{\overline {\cal P}}_{II}(v,u)\bigg|_{v=1}
= \frac{{\overline \psi}_{II}(u)}{(1-u)(1-{\overline \psi}_{II}(u))}, \qquad  |u|< 1.
\eeq  
\end{Proposition}
%
%
%
%
%
%
%
 %
%
%
%
%
%
%
%
%
%
\begin{Proposition}\label{propprop2}
 We have
 \beq
\label{second_moment-GF}
 \overline{\mathbb{E}  N^2_{II}}(u) = 
 \left[\frac{\partial^2}{\partial v^2}+\frac{\partial}{\partial v}\right]{\overline {\cal P}}_{II}(v,u)\bigg|_{v=1}
 = \frac{{\overline \psi}_{II}(u)[1+{\overline \psi}_{II}(u)]}{(1-u)[1-{\overline \psi}_{II}(u)]^2},  \qquad  |u|< 1. 
\eeq
\end{Proposition}
%
%
%
%
%
%
%
%
%
%
%
The following representation of (\ref{conditional-process}) 
\beq
\label{NI-renewal}
M(t)=  \Theta(\Delta T -1-t)\sum_{\ell=1}^{\infty}\Theta(t-J^{II}_{\ell}) + 
\Theta(t-\Delta T )\sum_{\ell=1}^{\infty}\Theta(\Delta T -J^{II}_{\ell})
\eeq
will be useful in the next theorem that gives the  state probabilities of $M(t)$.
\begin{thm}
\label{first-theorem}
Let $M(t)$ be defined as in \eqref{conditional-process}, $N_{II}(t)$ a RP of type II and $\Delta T$ an independent non-defective stopping time.
The PDF of $M(t)$ reads
\begin{eqnarray}
  \label{state_M} P_m(t) &:=& \mathbb{P}[M(t)=m]  \nonumber \\
   & = &  \Phi_S^{(0)}(t)\Phi_{II}^{(m)}(t)+ \sum_{r=1}^{t} \psi_S(r)\Phi_{II}^{(m)}(r)     \nonumber \\
  &=&\sum_{r=1}^{\infty} \psi_S(r) \left(\Phi_{II}^{(m)}(t) \Theta(r-t-1)+ \Phi_{II}^{(m)}(r)\Theta(t-r)\right)
\end{eqnarray}
where $m\in \mathbb N_0$, 
 $\Phi_{II}^{(m)}(r)=\mathbb{P}[N_{II} (r)= m]$.
 \end{thm}
\begin{proof}
First consider that 
\beq
\label{state_prob_M}
%
%
P_n(t)= \mathbb{P}[M(t)=n, \Delta T >t ] + \mathbb{P}[M(t)=n, \Delta T \leq t] ,
\eeq
with $n, t \in \mathbb{N}_0$ and initial condition $P_n(t)\big|_{t=0} = \delta_{n,0}$.
In this relation the event $\{M(t)=n\}$ may occur either if $N_{II}(t)=n$ and is not stopped yet (first term), 
or when it is stopped and $N_{II}(\Delta T)=n$ (second term). 
Hence, the first term writes
\beq
\label{conditional-probNII}
\ds 
\mathbb{P}[N_{II}(t)=n, \Delta T >t] 
%
= 
\mathbb{E}\mathbf{1}_{\{N_{II}(t)=n\}}\mathbf{1}_{ \Delta T >t } 
=
\mathbb{E} \Theta(J_n^{II},t,J_{n+1}^{II})  \mathbb{E} \Theta(\Delta T-1-t)  
\eeq
by the independence of $N_{II}(t)$ and 
$\Delta T$.
%
%
Here the probability
\begin{align}
\mathbb{P}[N_{II}(t)=n]=\mathbb{E} \Theta(J_n^{II},t,J_{n+1}^{II})  = \mathbb{P}[M(t)=n|\Delta T > t ]
\end{align}
is in fact the
conditional probability that $\{M(t)=n\}$ given $\Delta T > t$

%
%
%
%
%
%
The second term in (\ref{state_prob_M}) is the joint probability that the process is absorbed and counts $n$ events,
namely
\beq
\label{second_term_M}
\begin{array}{clr} 
\ds H_n(t) := \ds 
\mathbb{P}[M(t)= n, t\geq \Delta T] 
& =  \mathbb{E} \Theta(J_n^{II},\Delta T,J_{n+1}^{II}) 
\Theta(t-\Delta T)  \\ \\ &  =\ds  \sum_{r=1}^{t} \mathbb{P}[\Delta T=r] 
\mathbb{P}[N_{II}(r)=n] = \ds \sum_{r=1}^{t} \psi_S(r)\Phi_{II}^{(n)}(r)  &
\end{array}
\eeq
where we interpret $\Phi_{II}^{(n)}(r)=\mathbb{P}[M (t)= n|\Delta T=r]$, $t \ge r$, as the conditional probability that $M(t)$ counts $n$ events given that $\Delta T=r$.
All the terms together complete the proof.
\end{proof}
%
%
%
%
%
%
%
%
%
%
%
%
%
%
We observe that $P_m(t) = 0$ for $m>t$ inheriting this feature from $N_{II}(t)$ and the initial condition $P_m(t)\big|_{t=0}=\Phi_{II}^{(m)}(t))\big|_{t=0}=\delta_{m,0}$.
Then we can define the GF (state polynomial of $M(t)$ which is of degree $t$) 
\beq
\label{statepoly_M}
\begin{array}{clr} 
\ds 
\Pi(v,t) 
&= \sum_{m=0}^t P_m(t) v^m & \\ \\
 & =\ds  \sum_{r=1}^{\infty} \psi_S(r) \left\{{\cal P}_{II}(v,t) \Theta(r-t-1)+  
 {\cal P}_{II}(v,r)\Theta(t-r)\right\}. &
 \end{array}
\eeq
From the state polynomial we obtain all the moments of $M(t)$. In particular, recalling Proposition \ref{propprop}, the expected number of arrivals at time $t$ is
\beq
\label{expect_M}
\begin{array}{clr}
\ds \mathbb{E} M(t)  &  = \ds \mathbb{E} \Theta(\Delta T -1-t) \mathbb{E} N_{II}(t)  +
\mathbb{E}\Theta(t-\Delta T )N_{II}(\Delta T )  & \\ \\ &   = \ds 
\sum_{r=1}^{\infty} \psi_S(r) \left\{ \Theta(r-t-1)\mathbb{E} N_{II}(t) +  
 \Theta(t-r)\mathbb{E} N_{II}(r)  \right\}. & 
 \end{array}
\eeq
Analogously, from Proposition \ref{propprop2}, we deduce  the second moment as
\beq
\label{second_moment}
\mathbb{E} M^2(t)  =   
%
%
\sum_{r=1}^{\infty} \psi_S(r) \left\{ \Theta(r-t-1)\mathbb{E}N_{II}^2(t) +  
 \Theta(t-r) \mathbb{E} N_{II}^2(r) \right\}.
\eeq
For the moments of order $\ell \in \mathbb{N}$ we can write
\beq
\label{higher_moment}
\mathbb{E} M^{\ell}(t)  = 
\mathbb{E} \Theta(\Delta T-1-t)  \mathbb{E} N_{II}^{\ell}(t)  
+ \mathbb{E} \Theta(t-\Delta T) N_{II}^{\ell}(\Delta T) 
\eeq
as a consequence of representation (\ref{conditional-process}) and by using the properties of the indicator functions.
We point out that the results of this section, in particular 
Theorem \ref{first-theorem}, 
%
%
hold if the stopping time $\Delta T$ is non-defective.
Section \ref{defective_psiS} will consider the defective case.
%
%
%
%
%
%

\subsubsection{Asymptotic behavior}
It is useful to consider now the GF of the probability $H_n(t)$ defined in (\ref{second_term_M}) which writes
\beq
\label{Gf_conditional_prob}
\begin{array}{l} 
\ds {\overline H}_n(u)  =
\sum_{r=1}^{\infty} \psi_S(r)\Phi_{II}^{(n)}(r) \sum_{t=r}^{\infty}u^t   
= \ds \frac{1}{1-u}\sum_{r=1}^{\infty} \psi_S(r)\Phi_{II}^{(n)}(r) u^r.
\end{array}
\eeq
This relation involves
$
\chi_n(t)= \psi_S(t)\Phi_{II}^{(n)}(t)
$
which is a defective PDF with respect to $t$ and also with respect to $n$, see (\ref{incomplete_normalization}). 
The large time limit $H_n(\infty)$ gives the probability to observe a given number $n$ of events in the almost surely absorbed process.
From this we can get the large time limit of the expected number of arrivals in the almost surely absorbed process as
\beq
\label{stopped_process}
\mathbb{E} \, M(\infty) \,  =\sum_{m=0}^{\infty} m H_m(\infty).
\eeq
A sufficient condition for the finiteness of \eqref{stopped_process} is given in the following proposition.
\begin{Proposition}\label{senzanome}
    Let $\mathbb{E}\Delta T <\infty$. Then 
    \beq
        \mathbb{E} M(t)
        %
      %
      \leq  C' t^{-\nu} 
        +
        \sum_{r=1}^t  r \psi_S(r), \hspace{1cm} 
        C' >0, \hspace{0.5cm} \nu \in (0,1), 
    \eeq
    and $\mathbb{E}M(\infty) <\infty$.
\end{Proposition}
\begin{proof}
    From \eqref{expect_M},
    \beq
        \label{EM}
        \mathbb{E} M(t)  = \Phi_S^{(0)}(t)\mathbb{E} N_{II}(t)  +
        \sum_{r=1}^t\psi_S(r) \mathbb{E} N_{II}(r).
    \eeq
    By hypothesis, $\Delta T$ has a finite mean. Hence, the PDF $\psi_S(t)$ tends to zero at least as\footnote{We denote with $\sim$ asymptotic equality.} $\psi_S(t) \sim C t^{-2-\nu}$, $\nu \in (0,1)$, and the survival probability at least as $\Phi_S^{(0)}(t) \sim  C' t^{-1-\nu}$ where $C, C' >0$ are constants. Thus we have, for large values of $t$,
    \beq
        \label{upper-bound}
        \mathbb{E} M(t) \leq  C' t^{-\nu -1} t +
        \sum_{r=1}^t  r \psi_S(r),
    \eeq
    where we have used that
    $\mathbb{E} N_{II}(r)\leq t$, for every $r \le t$.
    Hence,
    \begin{align}
         C' t^{-\nu -1} t +
        \sum_{r=1}^t  r \psi_S(r) \sim C't^{-\nu} + \mathbb{E} \Delta T.
    \end{align}
    Moreover,
    \beq
        \label{limiting_M}
        \mathbb{E} M(\infty)  = \mathbb{E} N_{II}(\Delta T)  = \sum_{r=1}^{\infty}\psi_S(r) \mathbb{E} N_{II}(r) < \infty.
    \eeq 
\end{proof}

\begin{rmk}
 For non-defective $\Delta T$ with finite $\mathbb{E}\Delta T$
        the resetting term decays at least as $\Phi_S^{(0)}(t)\mathbb{E} N_{II}(t) \sim
    C' t^{-\nu }$ for large values of $t$. 
\end{rmk}

\begin{cor}
    Under the hypotheses of Proposition \ref{senzanome}, $M(t)$ is of type I. 
\end{cor}
Now, we look for conditions on $\Delta T$
such that the large time limit of the moments (\ref{higher_moment}) of any order of $M(t)$ remains finite. Indeed,
\beq
\label{ell-moments}
\begin{array}{clr}
\ds \mathbb{E} M^{\ell}(\infty)  & = \ds \lim_{t\to\infty}\left\{ \Phi_S(t) 
\mathbb{E} N_{II}^{\ell}(t)  + \sum_{r=1}^{t} \psi_S(r) \mathbb{E} N_{II}^{\ell}(r)  \right\}
& \\ \\  & \ds \leq \lim_{t\to\infty}  t^{\ell} \Phi_S(t) + \sum_{r=1}^{\infty} r^{\ell} \psi_S(r), \hspace{1cm} \forall \ell \in \mathbb{N} &  
\end{array}
\eeq 
where we have used $N_{II}(t) \leq t$ almost surely.
Formula \eqref{ell-moments} is finite if $\mathbb{E}\Delta T^\ell$ is finite which, in turn, implies $t^{\ell}\Phi_S(t) \to 0$  as $t\to \infty$,
i.e. $\Phi_{S}(t) \to 0$ at least geometrically.

\begin{rmk}
If the stopped process $N_{II}(t)$ is Bernoulli of parameter $p$ we have the relations (see (\ref{ell-moments}); 
see also \cite{Grimmet_Stirzacher}, Theorem 25, page, 153)
%
%
\beq
\label{simple-relations}
 \mathbb{E} M(\infty)   =   p \mathbb{E} \Delta T  , 
 \qquad  \mathbb{E} M^2(\infty)    = p^2 \mathbb{E} \Delta T^2  +pq \mathbb{E} \Delta T  ,
\eeq
\label{non_def_stop_time}
\end{rmk}
%
%
provided that $\mathbb E\Delta T < \infty$.
%
%
%
%

\subsubsection{Geometrically distributed $\bm{\Delta T}$} 
\label{geometric_stop}
We discuss now the case when $M(t)$ is of type I with 
non-defective geometric absorbing time $\Delta T$
with pdf 
$\psi_S(t)=pq^{t-1}$, $t \in \mathbb{N}$ with $p \in (0,1)$ and $q=1-p$.
Recall the expected 
stopping time $\mathbb{E} \Delta T = 1/p$.
Concerning the applications on random walks described in Section \ref{RWs} we are interested in the following large time limit of the state probabilities (\ref{state_M}). 
%
%
%
%
\beq
\label{long-timePm}
\begin{array}{clr}
\ds P_m(\infty) & = \ds  \lim_{t\to \infty} \left( q^t \Phi_{II}^{(m)}(t) + \sum_{r=1}^{t} pq^{r-1} \Phi_{II}^{(m)}(r) \right)  = \frac{p}{q}\left(-\delta_{m,0}+\sum_{r=0}^{\infty} q^r \Phi_{II}^{(m)}(r)\right)  & \\ \\ 
 & = \ds 
\left\{ \begin{array}{l} \ds  \frac{q -{\overline \psi}_{II}(q)}{q} , \hspace{1cm} m=0  \\ \\ \ds
\frac{1-{\overline \psi}_{II}(q)}{q} [{\overline \psi}_{II}(q)]^m  ,  \hspace{1cm} m >0. 
\end{array}\right. &
\end{array}
\eeq
Let us consider the  case $m=0$ more closely. We denote here with $\Delta t_{II}$ the arrival time of the first event of $N_{II}(t)$.
We have that
\beq
\label{P_0_2nd_derivation}
\begin{array}{clr}
\ds \ds P_0(\infty) & =
\mathbb{P}[\Delta T < \Delta t_{II}]
& \\
&   = \ds \sum_{k=1}^{\infty}\mathbb{P}[\Delta t_{II}=k] \mathbb{P}[\Delta T < k] &  \\
 & = \ds \sum_{k=2}^{\infty} \psi_{II}(k) \sum_{r=1}^{k-1} \psi_S(r) & \\ 
 &
 = \ds \sum_{k=1}^{\infty} \psi_{II}(k)[1-q^{k-1}]  
 =  \frac{q-{\overline \psi}_{II}(q)}{q}, &
 \end{array}
\eeq
where $ \mathbb{P}[\Delta T < 1] =0$.
The type I nature of the process is also confirmed by 
\beq
\label{infty}
\Lambda_M = \lim_{m_0 \to \infty} \mathbb{P}[M(\infty) >m_0] = \lim_{m_0 \to \infty} \sum_{m=m_0+1}^{\infty} P_m(\infty)=   \lim_{m_0 \to \infty} \frac{[{\overline \psi}_{II}(q)]^{m_0+1}}{q} =0.
\eeq
The case $q \to 1-$ represents the type II limit ($\Lambda_M=1$) where 
$N_{II}(t)$ is almost surely never stopped. On the other hand in the case $q \to 0+$ the process is immediately killed at $t=1$.
In this case $P_0(\infty) =1-\alpha_1$, $P_1(\infty) = \alpha_1$, $P_m(\infty) =0$, $m\geq 2$, for some 
$\alpha_1 = \psi_{II}(t)\big|_{t=1} \in [0,1]$.
%
%

From (\ref{long-timePm}) we derive the infinite-time limit of state polynomial (\ref{statepoly_M}), namely
%
%
\beq
\label{GF_Pm}
\Pi(e^{-\lambda},\infty) = 
\sum_{m=0}^{\infty} e^{-\lambda m}P_m(\infty) = -\frac{p}{q}+ \frac{1-{\overline \psi}_{II}(q)}{q[1-e^{-\lambda} \, {\overline \psi}_{II}(q)]}  , \hspace{1cm} \lambda \geq 0                        
\eeq
%
%
where we have set $v=e^{-\lambda}$ to obtain the moment GF. Observe that
$ \Pi(e^{-\lambda},\infty)\big|_{\lambda=0} =1 $ reflecting normalization of the state probabilities $P_m(\infty)$ at infinite times.

Let us now consider the large time limit of the moments (\ref{ell-moments}) by using the relation
\beq
\label{geoNs}
\mathbb{E} M^{\ell}(\infty) = (-1)^{\ell} \frac{d^{\ell}}{d\lambda^{\ell}} \Pi(e^{-\lambda},\infty)\bigg|_{\lambda=0}  =
\frac{p}{q}\sum_{t=0}^{\infty} 
\mathbb{E} N_{II}^{\ell}(t) u^t \bigg|_{u=q} 
\eeq
where we observe the GFs of the moments of $N_{II}(t)$ come into play.
For $\ell=1$ we get the following explicit relation for the expected number of events (\ref{limiting_M}):
\beq
\label{limit_M_Ber}
\mathbb{E} M(\infty) = \frac{p}{q}
\overline{\mathbb{E}  N_{II}}(u) \bigg|_{u=q} = \frac{{\overline \psi}_{II}(q)}{q[1-{\overline \psi}_{II}(q)]}. 
\eeq
For the second moments we get (cf.\ (\ref{second_moment-GF}) and (\ref{geoNs}))
\beq
\label{expect_2ndM}
\mathbb{E} M^2(\infty)
=
\frac{1}{q} \frac{{\overline \psi}_{II}(q)[1+{\overline \psi}_{II}(q)]}{[1-{\overline \psi}_{II}(q)]^2} = \mathbb{E} M(\infty) \frac{1+{\overline \psi}_{II}(q)}{1-{\overline \psi}_{II}(q)}.
\eeq
The asymptotic variance of $M(\infty)$ thus reads
\beq
\label{variance-bernoulli-stopped}
\mathbb{V}\text{ar} M(\infty) 
= \frac{{\overline \psi}_{II}(q)(q - p{\overline \psi}_{II}(q))}{q^2[1-{\overline \psi}_{II}(q)]^2}
.
\eeq
Note that the non-negativity of \eqref{variance-bernoulli-stopped} follows from
${\overline \psi}_{II}(q) \leq q$ for every $q \in [0,1]$. 
If
$p \to 1-$ 
then $M(t)$ is immediately stopped at $t=1$ with $M(t)=N_{II}(1)$  for all $t \geq 1$. 
Hence, necessarily,
\beq
\label{tone}
\lim_{p \to 1-}\mathbb{E} M(\infty)
=\lim_{p \to 1-}\mathbb{E} M^2(\infty) 
= \alpha_1
\eeq
\beq
\label{varP1}
\lim_{p \to 1-}\mathbb{V}\text{ar} M(\infty) =\alpha_1(1-\alpha_1).
\eeq
The state probabilities \eqref{state_M} take here the form
\beq
\label{state_Ber_sto}
P_m(t) = \mathbb{P}[M(t)=m]= q^t \Phi_{II}^{(m)}(t) + \frac{p}{q}\sum_{r=1}^t q^r \Phi_{II}^{(m)}(r)
\eeq
and it is useful to evaluate their GF:
\beq
\label{GFPm}
\begin{array}{clr}
\ds {\overline P}_m(u) & = \ds  \sum_{t=0}^{\infty}\Phi_{II}^{(m)}(t)(qu)^t +\frac{p}{q} 
\sum_{t=0}^{\infty} u^t \left( -\delta_{m,0} +\sum_{r=0}^t q^r \Phi_{II}^{(m)}(r) \right)  & \\ \\
 & = \ds {\overline \Phi}_{II}^{(m)}(qu) +\frac{p}{q(1-u)}\left( -\delta_{m,0} +  {\overline \Phi}_{II}^{(m)}(qu)\right)  & \\ \\ 
&  = \ds \frac{1}{q} \frac{1-{\overline \psi}_{II}(qu)}{(1-u)}[{\overline \psi}_{II}(qu)]^m  -\frac{p}{q(1-u)}\delta_{m,0}.
\end{array}
\eeq
Note that in \eqref{GFPm} we have
\begin{align}
{\overline \Phi}_{\mathcal{R}_q}^{(m)}(u)=  \frac{1-{\overline \psi}_{II}(qu)}{(1-u)} [{\overline \psi}_{II}(qu)]^m=  \sum_{m=0}^{\infty} \sum_{t=0}^{\infty} u^t \mathbb{P}[{\cal R}_q(t) = m],
\end{align}
that is the GF of the state probabilities of an auxiliary type I renewal process, say ${\cal R}_q(t)$,
with defective waiting time PDF $\psi_q(t)=q^t\psi_{II}(t)$. The state probabilities of ${\cal R}_q(t)$
are such that
\beq
\label{typeI_renewal_state}
\mathbb{P}[\mathcal{R}_q(\infty)=m] =\lim_{u\to 1-} (1-u){\overline \Phi}_{\mathcal{R}_q}^{(m)}(u) = 
[1-{\overline \psi}_{II}(q)][{\overline \psi}_{II}(q)]^m =\mathcal{P}\mathcal{Q}^m  ,\hspace{1cm} m \in \mathbb{N}_0,
\eeq
with $\mathcal{Q} = {\overline \psi}_{II}(q)$ which is indeed a geometric distribution supported on $\mathbb{N}_0$.
From \eqref{GFPm} we immediately retrieve
$ \ds P_m(\infty) = \lim_{u\to 1-} (1-u){\overline P}_m(u)$ given in (\ref{long-timePm}). The GF of \eqref{GFPm} then writes
\begin{equation}\label{GF_state_pol}
\ds {\overline \Pi}(v,u) = \ds  \sum_{m=0}^\infty v^m {\overline P}_m(u) 
  \ds = \frac{1}{q}{\overline \Pi}_{{\mathcal R}_q}(v,u) - \frac{p}{q(1-u)}    =\frac{1}{q}\frac{1-{\overline \psi}_{II}(qu)}{(1-u)[1-v{\overline \psi}_{II}(qu)]} - \frac{p}{q(1-u)}.
\end{equation}
We retrieve the limit $ \ds \lim_{u \to 1-}(1-u){\overline \Pi}(e^{-\lambda},u) = \Pi(e^{-\lambda},\infty)$ of expression (\ref{GF_Pm}). One further verifies that
${\overline \Pi}(v,u)\big|_{v=1} = \frac{1}{1-u}$,
that is the normalization of the state probabilities, and that $\ds {\overline \Pi}(v,u)\big|_{u=0}= 1$, consistent with $M(0)=0$ almost surely.
Note that in (\ref{GF_state_pol}) ${\overline \Pi}_{{\cal R}_q}(v,u) = \frac{1-{\overline \psi}_{II}(qu)}{(1-u)[1-v{\overline \psi}_{II}(qu)]} $ is the GF of the state polynomial of the renewal process $\mathcal{R}_q(t)$. 
What is the connection between $M(t)$ and the renewal process $\mathcal{R}_{q}(t)$?
To see this, consider the link between their state polynomials:
\beq
\label{M_with_rrenewal_part}
\Pi_{\mathcal{R}_{q}}(v,t) = p + q \Pi(v,t).
\eeq
The renewal structure of $\mathcal{R}_q(t)$ is evident as $\Pi_{\mathcal{R}_q}(v,t)$ satisfies the renewal equation
\beq
\label{renewal_eq_R}
\Pi_{\mathcal{R}_q}(v,t) = \mathbb{P}[\mathcal{R}_q(t)=0] + v \sum_{r=1}^t q^r \psi_{II}(r)\Pi_{\mathcal{R}_q}(v,t-r).
\eeq
Conversely, the state polynomial $\Pi(v,t)$ of $M(t)$ does not fulfill a corresponding renewal equation. This suggests that
$M(t)$ in fact is not a renewal process (see Section \ref{appendix_c} for further discussion).

\bigskip

\noindent  {\bf (i) Geometrically stopped Sibuya process} 
\\[2ex]
Here we consider $N_{II}(t)$ to be a Sibuya process and $\Delta T$  a non-defective geometric stopping time following the PDF 
$\psi_S(t) = p_Sq_S^{t-1}$ ($q_S=1-p_S$, $t\in \mathbb{N}$) in which $p_S$ is the associated success probability 
to stop the process in an independent Bernoulli trial at each time step. 
%
%
%

\bigskip

%
%
Let $\mu \in (0,1)$ be the parameter of the Sibuya distribution, which has the form
\beq
\psi_{Sib}(t) = (-1)^{t-1} \binom{\mu}{t}, \qquad t \in \mathbb{N},
\eeq
(see Section \ref{DSP}). We consider here the Sibuya process
being stopped by a type II Bernoulli process.  
Taking into account \eqref{newSibuya} with $\mathcal{Q}=1$ we have
\beq
\label{Ber-Sib}
\mathbb{E} M(\infty)  = \frac{1-p_S^{\mu}}{q_Sp_S^{\mu}} , \quad 
\mathbb{E} M^2(\infty)  = \frac{(1-p_S^{\mu})(2-p_S^{\mu})}{q_Sp_S^{2\mu}} , \quad \mathbb{V}\text{ar} M(\infty) = 
\frac{(1-p_S^{\mu})(q_S-p_S+p_S^{\mu+1})}{q_S^2p_S^{2\mu}},
\eeq
from which the limit $p_S \to 1-$ with $\alpha_1=\mu$ is immediately retrieved (see equations \eqref{tone}, \eqref{varP1}). 
We show the large time asymptotic values (\ref{Ber-Sib}) as a function of the stopping probability $p_S$ in Fig.~\ref{M-infinity}-(b). There, one can see that these quantities decrease monotonically with $p_S$. In Fig.~\ref{M-infinity}-(c) we compare the mean $\mathbb{E} M(\infty)$ with
the sample average values over $100$ realizations of $M(\Delta T)$. 
%
%
%
%
Note that the limit $p_S\to 0+$ represents the type II limit for which the Sibuya process is not stopped almost surely.
\begin{figure}
\centerline{
\includegraphics[width=0.7\textwidth]{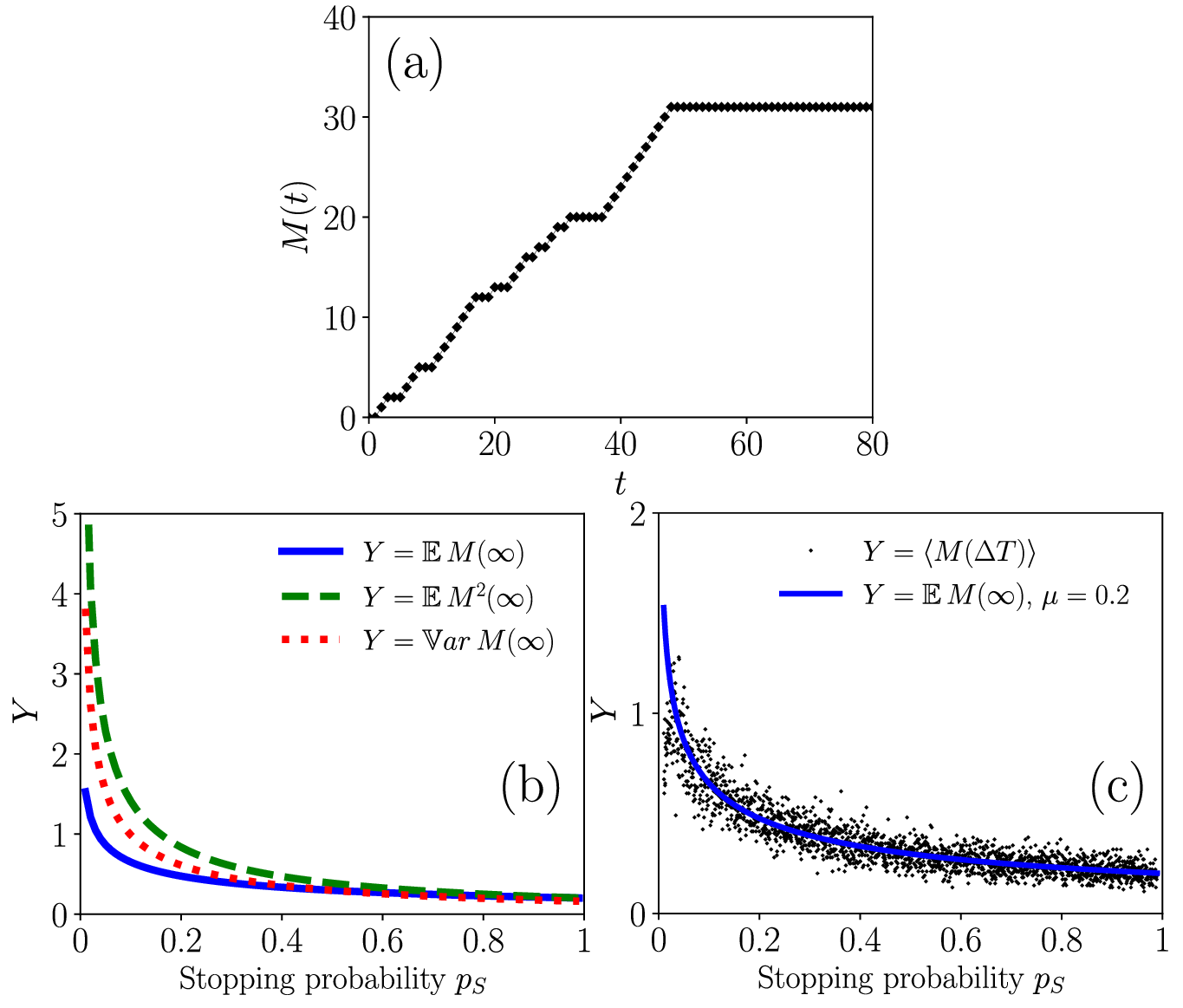}
}
%
%
%
\caption{
Geometrically stopped Sibuya process $N_{II}(t)$ with $\mu=0.2$.
The stopping time is non-defective.
(a): A sample path illustrating the qualitative behavior of the process ($p_S=0.01$) (b): Infinite time limits of (\ref{Ber-Sib}) plotted with respect to  $p_S$. 
(c): Sample path average ($100$ realizations) of $M(\Delta T)$ as a function of $p_S$. Each realization is run for maximum $t = 100$. The mean 
$\mathbb{E} M(\infty)$ from (\ref{Ber-Sib}) is represented in blue. 
}
\label{M-infinity}
\end{figure}

\bigskip
 
\noindent 
{\bf (ii) Geometrically stopped Bernoulli process}
\\[2ex]
Assume $N_{II}(t)$ to be a Bernoulli process with non-defective geometric stopping time $\Delta T$
with PDF $\psi_S(t)= p_Sq_S^{t-1}$ and let $\psi_{II}(t)= pq^{t-1}$ be the PDF associated to $N_{II}(t)$.
%
%
For this AP, which is clearly of type I for $p_S>0$,  considering \eqref{DFB-GF} for $\mathcal{Q} = 1$, we have that
\beq
\label{BB-infty}
\mathbb{E} M(\infty)  = \frac{p}{p_S} , \hspace{0.5cm} \mathbb{E} M^2(\infty)  =\frac{p[1+q_S(p-q)]}{p_S^2} \hspace{0.5cm}   \mathbb{V}\text{ar} M(\infty) = \frac{p[q+q_S(p-q)]}{p_S^2}.
\eeq
We depict the behavior of these quantities in Fig.~\ref{BerstopsB}.
\begin{figure}
\centerline{
\includegraphics[width=0.7\textwidth]{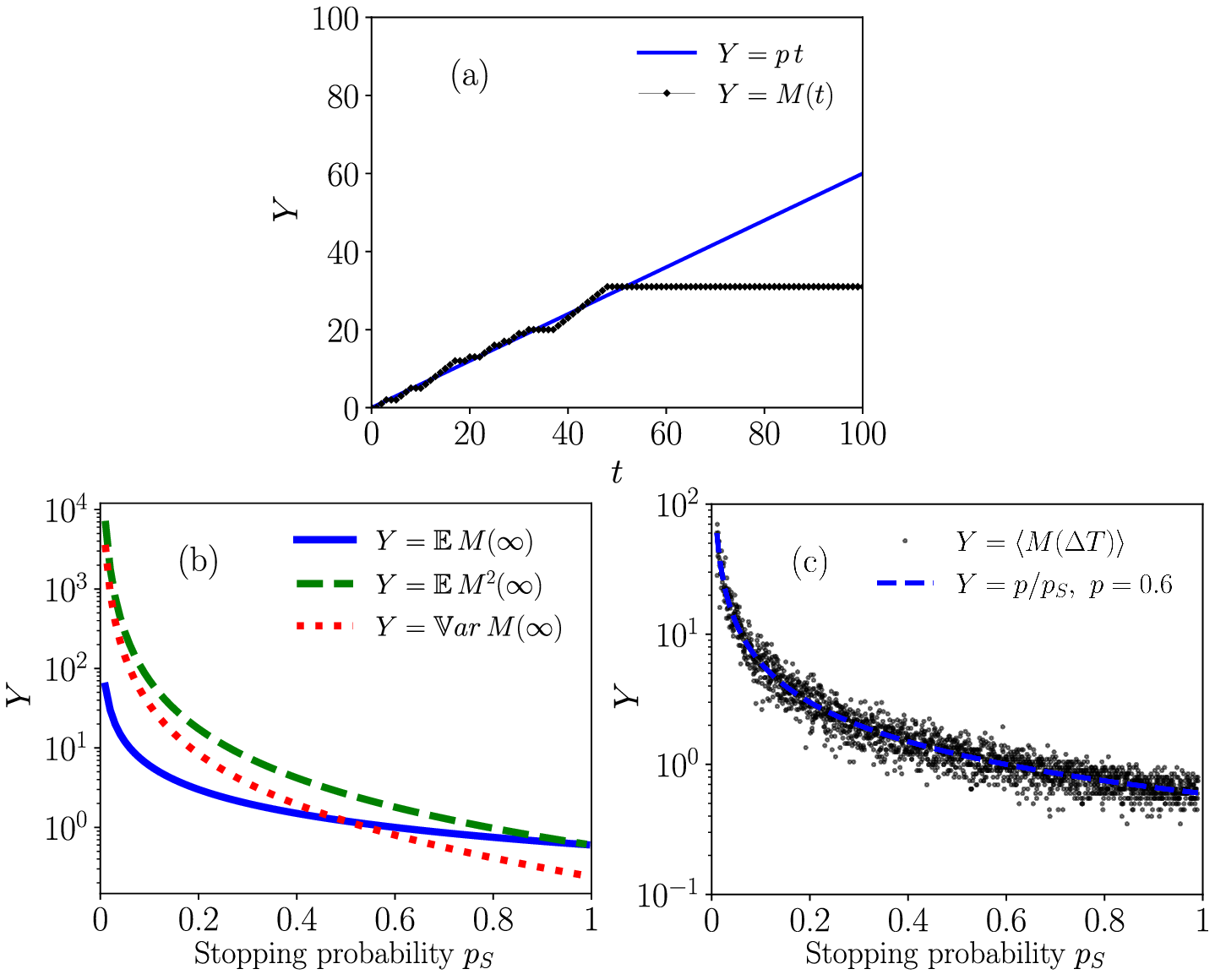}
}
\caption{A Bernoulli process $N_{II}(t)$ with  $p=0.6$ 
is geometrically stopped with success probability $p_S$. The geometric stopping time is non-defective. (a): A simulated sample path for $p_S=0.01$.
(b): The quantities in (\ref{BB-infty}) in log scale as functions of $p_S$. 
(c): Sample path average over $20$ realizations of $\langle M(\Delta T)\rangle$ as a function of $p_S$. Each realization is run for maximum $t = 5000$. The theoretical expected value is plotted in blue. 
}
\label{BerstopsB}
\end{figure}
%
%
In particular, Fig.~\ref{BerstopsB}-(c) shows a simulation of the average ($20$ realizations) of $M(\Delta T)$ as a function of $p_S$. The simulation matches with the theoretical value 
$\mathbb{E}\, M(\infty) = p/p_S$.

%
%
%
%
\subsubsection{Poisson-distributed $\bm{\Delta T}$}
\label{non-geo}
Here $N_{II}(t)$ is a Bernoulli process and 
$\Delta T$ is Poisson$(\lambda)$,  such that
$\psi_S (t) = \frac{\lambda^{t-1}}{(t-1)!} e^{-\lambda}$, $t \in \mathbb{N}$.
The stopping time GF is ${\overline \psi}_S(u) = ue^{\lambda(u-1)}$, $|u|<1$, thus from \eqref{limiting_M},
\beq
\label{Minfties}
\mathbb{E} M(\infty)  = p \frac{d}{du}[ue^{\lambda(u-1)}]\bigg|_{u=1}=  p(\lambda+1) , \hspace{0.3cm} \mathbb{E} M^2(\infty) = p\left[p(\lambda +1)^2 + \lambda + q \right].
\eeq
Hence, the variance of $M(\infty)$ reads
$\mathbb{V}\text{ar}M(\infty) = p(\lambda +q)$.
%
%
%
%
%

On the other hand, the limit which corresponds to stopping the Bernoulli process at $t=1$ a.s.\ is $\lambda \to 0+$, 
where the formulas in (\ref{Minfties}) boil down to (\ref{tone}) and (\ref{varP1}).
These asymptotic values increase monotonically as the mean stopping time of $\Delta T$ increases, that is the broader $\psi_S(t)$ the longer $N_{II}(t)$ remains unstopped.

\subsection{Defective $\bm{\Delta T}$}
\label{defective_psiS}
So far we assumed that the stopping time, i.e.
$\psi_S(t)$ was non-defective. 
In this section we consider the absorbing
time $\Delta T$ to be a defective random variable.
%
%
Recall that
\beq
\label{survial_def_finitetime}
\Phi_S^{(0)}(t)= 
\mathbb{E}\Theta(\Delta T-1-t) = 1-\mathcal{Q}_S+\sum_{r=t+1}^{\infty}\psi_S(r)  =1-\sum_{r=1}^t \psi_S(r)
\eeq
with the large time limit
\beq
\label{survival_defective}
\Phi_S^{(0)}(\infty)= \lim_{t\to \infty} \mathbb{P}[\Delta T >t] =1- \sum_{r=1}^{\infty}\psi_S(r) =1-{\cal Q}_S , \hspace{1cm} {\cal Q}_S \in (0,1]
\eeq
denoting the probability that $\Delta T=\infty$.
%
Equation (\ref{survial_def_finitetime}) is the consequence of the following proposition that holds for possibly defective RVs supported on the positive integers.
\begin{Proposition}
\label{averaging_defective_varibal}
Let $\Delta T$ be a discrete random variable with PDF $\psi_S(r)$, and $(0,1] \ni {\cal Q}_S= \sum_{r=1}^{\infty} \psi_S(r)$. Then
\beq
\label{averaging_feature_defective}
\mathbb{E} f(\Delta T) = \sum_{r=1}^{\infty} \psi_S(r)f(r) + (1-{\cal Q}_S)f(\infty)
\eeq
for suitable measurable functions $f(r)$ where $f(\infty) = \lim_{r\to \infty} f(r)$. 
\end{Proposition}
This proposition accounts for the fact that a probability mass
$1-{\cal Q}_S$ is
located at infinity.
The non-defective case appears for ${\cal Q}_S=1$.
This relation yields infinity for the expected waiting time $\mathbb{E} \Delta T$ when $\Delta T$ is defective.
Considering $f(r)=u^r$, $|u| \le 1$, in (\ref{averaging_feature_defective}) yields
\beq
\label{mean_pwower}
\mathbb{E} u^{\Delta T} = \left\{\begin{array}{ll} {\overline \psi}_S(u) , & |u|<1 \\ \\
1 , & u=1 
\end{array}\right.
\eeq
which is discontinuous at $u=1$ for $\mathcal{Q}_S\ne 1$, namely
\begin{equation}\label{thermos}
    {\overline \psi}_S(u)\big|_{u=1} = \mathbb{E} u^{\Delta T}\big|_{u=1-} = {\cal Q}_S.
\end{equation}
It is then straightforward to recover the GF
\begin{equation}\label{tablet}
    {\overline \Phi}_S^{(0)}(u) = \frac{1 -\mathbb{E} u^{\Delta T}}{1-u} = \frac{1 - {\overline \psi}_S(u)}{1-u} , \hspace{1cm} |u|<1,
\end{equation}
which has the same structure as in the non-defective case (\ref{GF_state_II}) but this time with the property in \eqref{thermos}.
%
%
From \eqref{tablet} we recover the infinite-time limit (\ref{survival_defective}), namely
%
%
\beq
\label{recovere_asymp}
\Phi_S^{(0)}(\infty) = \lim_{u\to 1-} (1-u){\overline \Phi}_S^{(0)}(u) 
=1-{\cal Q}_S .
\eeq
Proposition \ref{averaging_defective_varibal} leads to the following extension of Theorem \ref{first-theorem} for the defective case.
\begin{thm}
\label{psi_s_defect_theorem}
Let $M(t)$ be defined as in \eqref{conditional-process}, and 
$\psi_S(r)$ be defective with $\sum_{r=1}^{\infty} \psi_S(r) = \mathcal{Q}_S \in (0,1]$. 
Then, the PDF of $M(t)$ reads
\begin{equation}
    \label{prob_mass_defect}
    \mathbb{P}[M(t)=m]= (1-\mathcal{Q}_S)\Phi_{II}^{(m)}(t) + \sum_{r=1}^{\infty} \psi_S(r) \left(\Phi_{II}^{(m)}(t) \Theta(r-t-1)+ \Phi_{II}^{(m)}(r)\Theta(t-r)\right).
\end{equation}
\end{thm}
\begin{proof}
    Recalling
    \begin{equation}
    \label{state_M_arrival}
        \mathbb{P}[M(t)=m]= \Phi_S^{(0)}(t)\Phi_{II}^{(m)}(t)+ \sum_{r=1}^{t} \psi_S(r)\Phi_{II}^{(m)}(r),
    \end{equation}
    the result follows from Theorem \ref{first-theorem} taking into account \eqref{survial_def_finitetime}.
\end{proof}
We observe that the normalization condition 
%
%
remains fulfilled when the stopping time is defective
\beq
\label{norma_defect_state}
\sum_{m=0}^{\infty}\mathbb{P}[M(t)=m] = \Phi_S^{(0)}(t) + \sum_{r=1}^{t} \psi_S(r) = 1
\eeq
by virtue of (\ref{survial_def_finitetime}) and $\sum_{m=0}^{\infty}\Phi_{II}^{(m)}(t)=1$.
The probability that $M(t)$ exceeds $m_0$, with $m_0\in \mathbb{N}_0$, then writes
\beq
\label{exceed_prob}
\begin{array}{clr} 
\ds \mathbb{P}[M(t) >m_0] & = \ds \mathbb{E} \Theta(\Delta T-1-t) \mathbb{E} \Theta(t-J_{m_0+1}) +
\mathbb{E} \Theta(t-\Delta T)\Theta(\Delta T-J_{m_0+1})    & \\ 
 & =\ds  \Phi_S^{(0)}(t)\mathbb{P}[N_{II}(t) >m_0] +\sum_{r=1}^t\psi_S(r)\mathbb{P}[N_{II}(r) >m_0]. &
 \end{array}
\eeq

The infinite time limit yields
\beq
\label{exceed_probinfinite-time}
\mathbb{P}[M(\infty) >m_0] = \Phi_S^{(0)}(\infty)\mathbb{P}[N_{II}(\infty) >m_0] + \sum_{r=1}^{\infty}\psi_S(r)\mathbb{P}[N_{II}(r) >m_0] 
\eeq
where in the second term we used
\begin{align}
    \lim_{t\to \infty}\mathbb{E}\Theta(t-\Delta T)f(\Delta T) =\lim_{t\to \infty}\left\{ \sum_{r=1}^{t}\psi_S(r)f(r) + \Theta(t-\infty)f(\infty) \right\} =\sum_{r=1}^{\infty}\psi_S(r)f(r)
\end{align}
with $\lim_{t\to \infty}\Theta(t-\infty)=\lim_{t\to \infty} \{\lim_{r\to \infty}\Theta(t-r)\}=0$ where first we let $r\to \infty$ while $t$ is kept constant to apply Theorem \ref{psi_s_defect_theorem} and then we consider $t\to \infty$.
Since $N_{II}(t)$ is of type II
we know that $\lim_{t\to\infty}\mathbb{P}[N_{II}(t) >m_0] = 1$. Since $\Phi_S^{(0)}(\infty)= 1-{\cal Q}_S$, the probability that $M(t)$ is never stopped reads
\beq
\label{Lambda-M}
\Lambda_M =\lim_{m_0\to \infty}\mathbb{P}[M(\infty) >m_0] = 1-{\cal Q}_S + \lim_{m_0\to \infty} \sum_{r=1}^{\infty}\psi_S(r)\mathbb{P}[N_{II}(r) >m_0].
\eeq
Now take into account that
$\mathbb{P}[N_{II}(t)>m_0] =0$ for $t \leq m_0$ as $N_{II}(t) \leq t$. Therefore,
\beq
\label{reconsider_prob_inf_t}
\sum_{r=1}^{\infty} \psi_S(r) \mathbb{P}[N_{II}(r)>m_0] = \sum_{r=m_0+1}^{\infty} \psi_S(r) \mathbb{P}[N_{II}(r)>m_0]  \leq \sum_{r=m_0+1}^{\infty} \psi_S(r) = {\cal Q}_S-\sum_{r=1}^{m_0} \psi_S(r)
\eeq
and hence
\beq
\label{LambdaM_general}
\Phi_S^{(0)}(\infty) =
1-{\cal Q}_S \leq \Lambda_M \leq \lim_{m_0\to \infty}\bigg[ 1 - \sum_{r=1}^{m_0} \psi_S(r) \bigg]= 
\lim_{m_0 \to \infty} \Phi_S^{(0)}(m_0) = 1-{\cal Q}_S 
\eeq
thus $\Lambda_M=1-{\cal Q}_S$. 
The previous result suggests the following interpretation.
\begin{itemize}
    \item {\it  If the stopping time $\Delta T$ is defective, 
    %
    %
    %
    %
    then $M(t)$ is an IAP, and it is never stopped with probability $\Lambda_M=1-{\cal Q}_S = \Phi_S^{(0)}(\infty) <1$;} 
    \item {\it If the stopping time $\Delta T$ is non-defective, 
    %
    %
    then we have $\Lambda_M=0$, $M(t)$ is stopped almost surely and hence it is of type I.}
    \item {\it If ${\cal Q}_S \to 0+$, one has that $M(t)$ is of type $II$ (see (\ref{conditional-process})) and $\Lambda_M \to 1-$.}
\end{itemize}
\subsubsection{Defective geometrically-distributed \bm{$\Delta T$}}
%
%
%
Here we assume the stopping time $\Delta T$ to be defective with $\psi_S(t) = \mathcal{Q}_S p q^{t-1}$ 
where $\mathcal{Q}_S \in (0,1]$.
%
%
%
%
%
%
Then, as mentioned above, $M(t)$ is an IAP which is not stopped with probability $\Lambda_M=1-\mathcal{Q}_S$.
If $\mathcal{Q}_S \to 1-$ then $M(t)$ is of type I, whilst if $\mathcal{Q}_S \to 0+$ then $M(t)$ if of type II.
The state probabilities (\ref{prob_mass_defect}) and their GF read 
\beq
\label{state_defect_berBer_sto}
P_m(t) = (\Lambda_M+\mathcal{Q}_S q^t) \Phi_{II}^{(m)}(t) + \mathcal{Q}_S\frac{p}{q}\left[-\delta_{m,0}+\sum_{r=0}^t q^r \Phi_{II}^{(m)}(r)\right],
\eeq
\beq
\label{GFPm_defective_Ber}
\ds {\overline P}_m(u)  = \Lambda_M \frac{1- {\overline \psi}_{II}(u)}{1-u}[{\overline \psi}_{II}(u)]^m   +\frac{\mathcal{Q}_S}{q(1-u)}\left([1-{\overline \psi}_{II}(qu)][{\overline \psi}_{II}(qu)]^m  - p\delta_{m,0}\right),
\eeq
where \eqref{state_defect_berBer_sto}
and \eqref{GFPm_defective_Ber} respectively specialize to (\ref{state_Ber_sto}) and (\ref{GFPm}) for $\mathcal{Q}_S=1$.

Of interest is the infinite time limit 
\beq
\label{inifinite_timelimitPms}
P_m(\infty)  = \lim_{u\to 1} (1-u){\overline P}_m(u) =\frac{\mathcal{Q}_S}{q}[1-{\overline \psi}_{II}(q)][{\overline \psi}_{II}(q)]^m  -\mathcal{Q}_S\frac{p}{q}\delta_{m,0}
\eeq
where the first term in (\ref{GFPm_defective_Ber}) necessarily yields zero and for $\mathcal{Q}_S=1$ 
(\ref{inifinite_timelimitPms}) recovers (\ref{long-timePm}).
The following feature appears noteworthy. 
Since for any finite $t$, 
%
%
\begin{align}
\lim_{t\to \infty} \sum_{m=0}^{\infty} P_m(t) = 1 \neq \sum_{m=0}^{\infty} P_m(\infty) =\mathcal{Q}_S.
\end{align}
The reason for the difference lies in relation (\ref{averaging_feature_defective}). 

Summing up (\ref{GFPm_defective_Ber}) over $m$ takes us to
the GF of the state polynomial 
${\overline \Pi}_{\text{\it IAP}}(v,u)= \sum_{t=0}^{\infty}u^t \mathbb{E} v^{M(t)}$ which can be written as
\beq
\label{state_pol_IAP}
{\overline \Pi}_{\text{\it IAP}}(v,u) = 
\Lambda_M {\overline {\cal P}}_{II}(v,u) + \frac{\mathcal{Q}_S}{q(1-u)}\left(\frac{1-{\overline \psi}_{II}(qu)}{1-v{\overline \psi}_{II}(qu)}-p\right) , \hspace{1cm} |v| \leq 1, |u|<1 
\eeq
where ${\overline {\cal P}}_{II}(v,t)$ is the GF (\ref{state_poly}) associated to $N_{II}(t)$. 
For $\mathcal{Q}_S \to 1-$, 
(\ref{state_pol_IAP}) reduces to (\ref{GF_state_pol}) for the type I process $M(t)$.

Of interest is the large time asymptotics 
\beq
\label{large_time_state}
\mathbb{E} v^{M(\infty)} = \lim_{t\to \infty} \Pi_{\text{\it IAP}}(v,t) = \lim_{u\to 1-}(1-u){\overline \Pi}_{\text{\it IAP}}(v,u) = \frac{\mathcal{Q}_S}{q}\left(\frac{1-{\overline \psi}_{II}(q)}{1-v{\overline \psi}_{II}(q)}-p\right),      \quad |v| \leq 1.
\eeq
The non-negativity of this expression can be easily seen from ${\cal \psi}_{II}(q) < q$, $q\in (0,1)$. This result remains true for complex $v=e^{i\varphi}$, $\varphi \in (0,2\pi)$, whereas for $v=1$ and taking into account (\ref{state_pol_IAP}), we have that the limit $\lim_{t \to \infty}\Pi_{\text{\it IAP}}(1,t) = 1$. Furthermore, (\ref{large_time_state}) retrieves (\ref{GF_Pm}) for $\mathcal{Q}_S=1$.
\\
In the following examples we will consider the first two moments of $M(t)$ whose GFs take the general forms
\beq
\label{fist_mom_DBP}
\overline{\mathbb{E} {M}}(u) = \frac{\partial}{\partial v} {\overline \Pi}_{\text{\it IAP}}(v,u)\bigg|_{v=1} =
\frac{1}{1-u} \left(\Lambda_M \frac{{\overline \psi}_{II}(u)}{1-{\overline \psi}_{II}(u)} + 
\frac{\mathcal{Q}_S}{q}\frac{{\overline \psi}_{II}(qu)}{1-{\overline \psi}_{II}(qu)} \right),
\eeq
\beq
\label{second_mom_DBP}
\overline{\mathbb{E}{M^2}}(u) = 
\Lambda_M\frac{{\overline \psi}_{II}(u)[1+{\overline \psi}_{II}(u)]}{(1-u)[1-{\overline \psi}_{II}(u)]^2}
+\frac{\mathcal{Q}_S}{q}\frac{{\overline \psi}_{II}(qu)[1+{\overline \psi}_{II}(qu)]}{(1-u)[1-{\overline \psi}_{II}(qu)]^2}.
\eeq
%
%
%

%
%
%
%

%
%
\paragraph{ Bernoulli process stopped by a defective geometrically-distributed 
\bm{$\Delta T$} }
%
%

\vspace{2mm} Let $\psi_S(t)= \mathcal{Q}_Spq^{t-1}$, $\mathcal{Q}_S \in (0,1]$. Let $N_{II}(t)$ be an independent  Bernoulli process with success probability $p_0$ (and $q_0=1-p_0$).
The state polynomial of 
$N_{II}(t)$ is known to be ${\cal P}_{II}(v,t)= (p_0v+q_0)^t$, $t\in \mathbb{N}_0$. Thus, the GF (\ref{state_pol_IAP}) reads
\beq
\label{GF_Dber_Ber}
 {\overline \Pi}_{{\text{\it IAP}}}(v,u) =  \frac{\Lambda_M}{1-u{\cal A}(v)} + 
\frac{\mathcal{Q}_S}{1-q{\cal A}(v)}\left(\frac{p{\cal A}(v)}{1-u} + \frac{1-{\cal A}(v)}{1-uq{\cal A}(v)}\right) 
\eeq
where we have set ${\cal A}(v)= p_0v+q_0 = {\cal P}_{II}(v,1)$.
The state polynomial then writes
\begin{align}
\label{GF_Dber_Ber_statepoly}
\ds \Pi_{{\text{\it IAP}}}(v,t)  &= \ds \mathbb{E} v^{M(t)} = (\Lambda_M+\mathcal{Q}_Sq^t){\cal A}(v)^t + \frac{\mathcal{Q}_Sp{\cal A}(v)}{1-q{\cal A}(v)}[1-q^t{\cal A}(v)^t] \notag \\
 & = \ds \Lambda_M{\cal A}(v)^t + \frac{\mathcal{Q}_S}{1-q{\cal A}(v)}\left(p{\cal A}(v)+(1-{\cal A}(v))q^t{\cal A}(v)^t\right)
\end{align}
with $\Pi_{{\text{\it IAP}}}(1,t)=1$ and $\Pi_{{\text{\it IAP}}}(v,0)=1$. Observe that the process $M(t)$ is not Markovian. Actually, it is not even in the non-defective case. Indeed, consider the type I limit $\mathcal{Q}_S=1$, i.e. when 
the geometric stopping time becomes non-defective.
That $M(t)$ is not Markovian can be seen from
$\Pi_{{\text{\it IAP}}}(v,t)\Pi_{{\text{\it IAP}}}(v,1) \neq  \Pi_{{\text{\it IAP}}}(v,t+1)$, that violates the Markov condition. 
In particular, $\Pi_{{\text{\it IAP}}}(v,1)= {\cal A}(v)$ is a function of the Bernoulli parameter $p_0$ of $N_{II}(t)$ only, and not of the stopping probability $p$.
This is because no matter whether $N_{II}$ is stopped at $t=1$ or not, 
we have $M(1)=N_{II}(1)$ a.s., and the latter is independent of $\Delta T$,
%
%
see (\ref{conditional-process}). 
This does not hold true for $t\geq 2$, as $\Pi_{{\text{\it IAP}}}(v,t)$ is a function of both the parameters of $N_{II}$ 
and of $\Delta T$.
%
%
%
%
Hence ${\cal A}(v)^2 \neq \Pi_{{\text{\it IAP}}}(v,2)$, i.e.\ $M(t)$ is not Markovian (except for the special case $\mathcal{Q}_S=0$).
\\
For $|v|<1$ ($|{\cal A}(v)|<1$) we have the stationary value
\begin{align}\Pi_{{\text{\it IAP}}}(v,\infty) = \frac{\mathcal{Q}_S p{\cal A}(v)}{1-q{\cal A}(v)}
\end{align}
%
{which is, indeed, a NESS,  for $\mathcal{Q}_S \in (0,1]$, }
see (\ref{large_time_state}). If $\mathcal{Q}_S=p=1$, 
%
%
($M(t)$ is stopped at $t=1$ a.s.) the stationary value 
is taken already at $t=1$ with $\Pi_{{\text{\it IAP}}}(v,\infty)= \mathbb{E} v^{N_{II}(1)} = \mathcal{A}(v)$.
Clearly (\ref{GF_Dber_Ber_statepoly}) is strictly positive for $v \in [0,1]$ and
\beq
\label{Piv0}
\Pi_{{\text{\it IAP}}}(v,t)\bigg|_{v=0} = \mathbb{P}[M(t)=0] = \Lambda_M q_0^t + \frac{\mathcal{Q}_S}{1-q q_0}\left(pq_0+p_0q^t q_0^t\right)
\eeq
consistently with (\ref{state_defect_berBer_sto}) and considering the initial condition $\mathbb{P}[M(0)=0]=1$.
We point out that the state polynomial (\ref{GF_Dber_Ber_statepoly})
is an absolutely monotonic (AM) and non-decreasing function with respect to $v$, i.e. for $n \in \mathbb{N}_0$,
\beq
\frac{\partial^n}{\partial v^n} \Pi_{{\text{\it IAP}}}(v,t) \geq 0.
\eeq
From the relation between AM functions and CM functions \cite{bernstein1929}, the non-negativity of the state probabilities and moments of $M(t)$ follows. The absolute monotonicity can be easily verified in (\ref{GF_Dber_Ber_statepoly})
as $\frac{1-q^t{\cal A}(v)^t}{1-q{\cal A}(v)} = \sum_{r=0}^{t-1} q^r{\cal A}(v)^r$ and taking into account that $q^r{\cal A}(v)^r$ belongs to the AM class.
Now, let us consider the first two moments. We have
\beq
\label{IAM_1st_moment}
\mathbb{E} M(t) =  \Lambda_M p_0t+\frac{\mathcal{Q}_Sp_0}{p}(1-q^t):=\Lambda_M p_0t +\mathcal{Q}_SB(t), \qquad \mathbb{E} M(0) =0.
\eeq
For large times the dominating term in the first moment is  $\Lambda_M p_0t$ as $N_{II}(t)$ is eventually not stopped with probability $\Lambda_M$. 
 In Fig.~\ref{Expected_M} we depict $\mathbb{E} M(t)$ from (\ref{IAM_1st_moment}) for several values of $\mathcal{Q}_S$ (colored curves) and we compare them with Monte Carlo simulations.
For large times $\mathbb{E} M(t)$ behaves linearly. In particular, for $\mathcal{Q}_S=1$ 
the finite asymptotic value $\mathbb{E} M(\infty)=p_0/p=3.5$ is approached.
%
%

%
%
\begin{figure}
%
%
\centerline{
\includegraphics[width=0.7\textwidth]{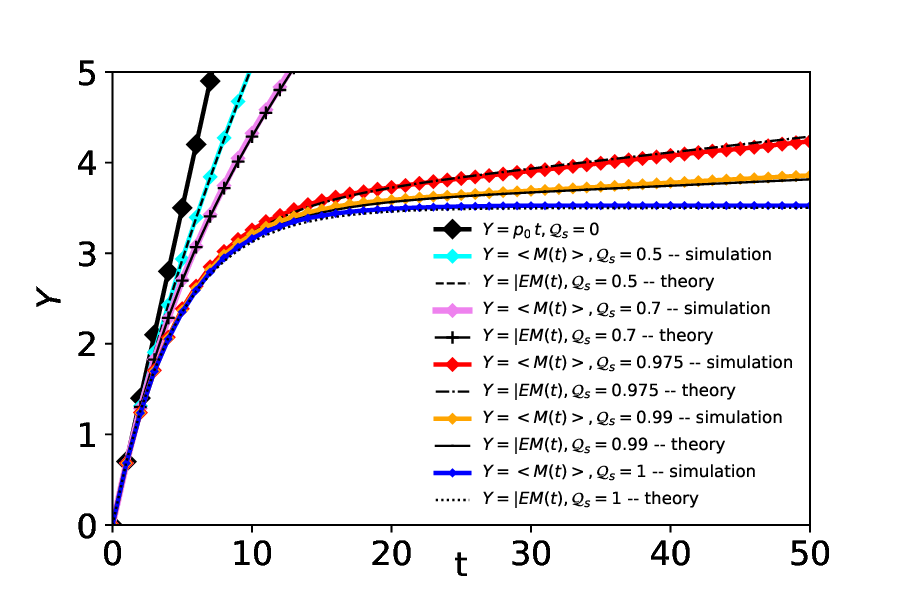}
}
\caption{ 
$\mathbb{E} M(t)$ from (\ref{IAM_1st_moment}) for several values of $\mathcal{Q}_S$ and with parameters $q_0=0.3$, $q=0.8$ (black curves). 
The colored curves represent sample averages of the process over $1500$ realizations obtained by Monte Carlo simulations.
}
\label{Expected_M}
\end{figure}
\noindent The second moment reads
\begin{eqnarray}
\label{M_second}
\ds \mathbb{E} M^2(t) &=& \ds \left(\frac{\partial^2}{\partial v^2} + \frac{\partial}{\partial v}\right) \Pi_{\textbf{\it IAP}}(v,t)\bigg|_{v=1} \notag \\
&=&\ds 
\Lambda_M(p_0^2t^2 +p_0q_0t) + \mathcal{Q}_S \left( \frac{2p_0^2q}{p^2}(1-q^t -ptq^{t-1}) + \frac{p_0}{p}(1-q^t) \right) \notag \\ 
&:=& \ds \Lambda_M (p_0^2t^2 +p_0q_0t) + \mathcal{Q}_S C(t). 
\end{eqnarray}
We are now ready to calculate the variance:
\begin{align}
\label{variance_defber}
\ds \mathbb{V}\text{ar}{M(t)} & =  \Lambda_M\left[\mathcal{Q}_Sp_0^2t^2+p_0q_0t\right] -2\Lambda_M\mathcal{Q}_Sp_0tB(t)+\mathcal{Q}_S\left[C(t)-\mathcal{Q}_SB^2(t)\right] \notag \\
 & = \ds C(t)-B^{\, 2}(t) +\Lambda_M[2B(t)(B(t)-p_0t) - C(t) + p_0^2t^2 +p_0q_0t] - \Lambda_M^{\, 2}\, (p_0t-B(t))^2.
\end{align}
The main feature is that the variance of $M(t)$ exhibits a $t^2$-behavior for large times (in contrast to the linear behavior of the unstopped Bernoulli emerging for $\Lambda_M=1$). This result can be interpreted that the defective Bernoulli stopping process introduces large fluctuations. 
In Fig. \ref{Var_plot} we depict  the variance (\ref{variance_defber}) for several values of the defect parameter $\mathcal{Q}_S$ starting from $\mathcal{Q}_S=1$
for which $M(t)$ is of type I and is eventually stopped, and the variance has a finite asymptotic value. 
For decreasing $\mathcal{Q}_S$, 
%
%
$\Delta T$ struggles to stop $N_{II}(t)$.
For $\mathcal{Q}_S \to 0+$ the linear behavior
of the  Bernoulli $N_{II}(t)$ emerges (black line).
In expression (\ref{variance_defber}) one can see that the coefficient $\Lambda_M\mathcal{Q}_S$ of the term $t^2$ has its maximum
value for $\Lambda_M=\mathcal{Q}_S=\frac{1}{2}$ where the intermediate feature of $M(t)$ is most pronounced and $\mathbb{V}\text{ar} M(t) \sim  \frac{p_0^2}{4} t^2$.
In general, the variance $\mathbb{V}\text{ar}M(t)$ is a parabolic function of $\Lambda_M$.
The value of $\Lambda_M$ maximizing $\mathbb{V}\text{ar} M(t)$ turns out to be 
\beq
\label{Max_Ps}
\Lambda_{M,\max}(t) = \frac{1}{2[p_0t-B(t)]^2}\left[p_0^2t^2+p_0q_0t -C(t) + 2B(t)(B(t)-p_0t)\right] , \qquad t \geq 2,
\eeq
where the asymptotic value $\Lambda_{M,\max}(t)$ is independent of $q$ and $q_0$ for $t \to \infty$ and rapidly approaches $1/2$ (Fig.~\ref{Psmax}). Hence, in this case we identify the IAP $M(t)$ with $\Lambda_M=1/2$ as the most fluctuating one in the long time limit.
We depict the time dependence of the variance for $\Lambda_M=1/2$ in Fig.~\ref{Var_plot} (cyan curve), and of $\Lambda_{M,\max}(t)$ in Fig.\ \ref{Psmax}.
This suggests that IAPs exhibit  large fluctuations in the large time limit.
\begin{figure}[H]
\centerline{\includegraphics[width=0.9\textwidth]{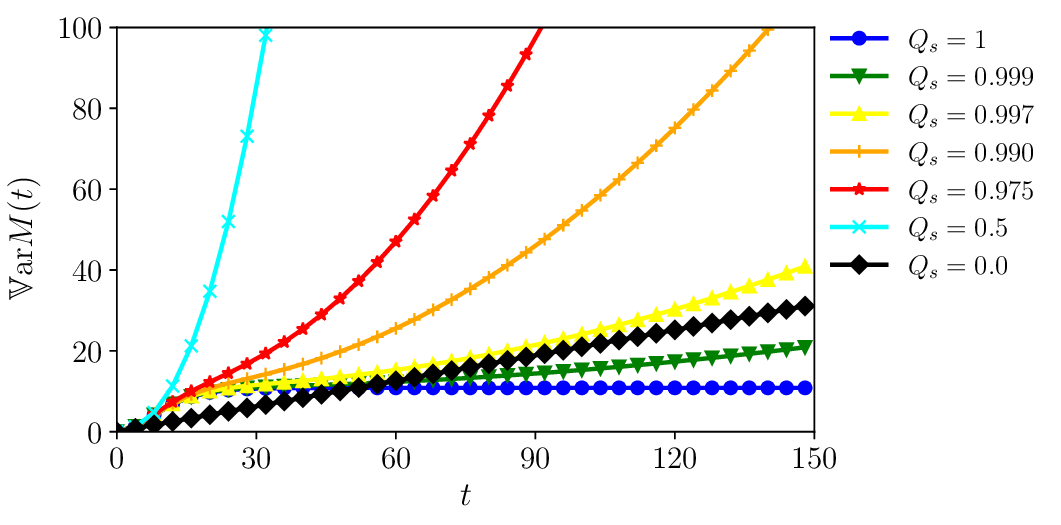}}
\vspace{-3mm}
\caption{The variance (\ref{variance_defber}) as a function of time for several choices of $\mathcal{Q}_S$. We have chosen here in all curves $q_0=0.3$ as the parameter of the  Bernoulli $N_{II}(t)$ while the stopping time parameter $q=0.8$. For $\mathcal{Q}_S=1$ (blue curve) the AP $M(t)$ is of type I and the variance has the finite asymptotic value $\mathbb{V}\text{ar}M(\infty) = \frac{2p_0^2q}{p^2}+\frac{p_0}{q} \approx 10.85$. For $\mathcal{Q}_S=0$ we have a type II limit $M(t)=N_{II}(t)$ where the variance increases linearly 
For $\mathcal{Q}_S \in (0,1)$ the variance exhibits second-order large time asymptotics corresponding to large fluctuations which are most pronounced in the
case $\mathcal{Q}_S=1/2$ (cyan curve).}
\label{Var_plot}
\end{figure}
\begin{figure}
\centerline{\includegraphics[width=0.7\textwidth]{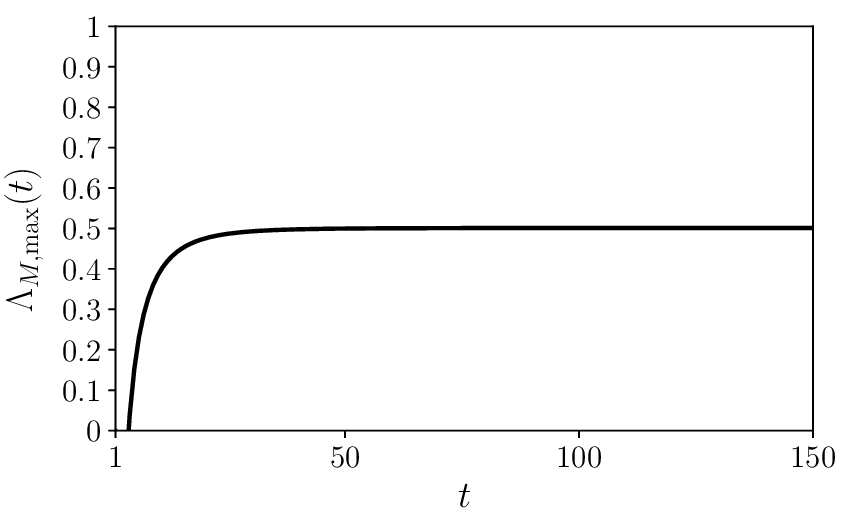}}
\vspace{-3mm}
\caption{ Time dependence of $\Lambda_{M,\max}(t)$ (formula (\ref{Max_Ps})) for which the fluctuations (\ref{variance_defber}) take a maximum. We observe that the asymptotic value $1/2$ is rapidly approached. The other parameters are $q_0=0.3$ and $q=0.8$ as in Fig. \ref{Var_plot}.}
\label{Psmax}
\end{figure}
%
%

%
%
\subsubsection{Bernoulli process stopped by a defective Sibuya \bm{$\Delta T$}}
\label{defective_Sibuya_stop}
Let $N_{II}$ be a Bernoulli process with parameter $p_0$ which is stopped by a defective Sibuya $\Delta T$, see (\ref{DSP_den}). 
%
%
%
%
%
%
The resulting AP $M(t)$ clearly is an IAP for $\mathcal{Q}_s \in (0,1)$.
Using relation (\ref{higher_moment}) and the survival probability (\ref{large-time-survival}) yields
\beq
\label{expected_events_def_sib_stop}
\mathbb{E} M(t) = \left [\mathcal{P}_s + \mathcal{Q}_s (-1)^{t} \binom{\mu-1}{t} \right] p_0t +p_0\mathcal{Q}_s\sum_{r=1}^t (-1)^{r-1} \binom{\mu}{r} r
\eeq
where, for $t\geq 1$,
$$
 (-1)^{t} \binom{\mu-1}{t}\, t =  \frac{1}{\Gamma(1-\mu)} \frac{\Gamma(t-\mu+1)}{\Gamma(t)} , \hspace{1cm} \sum_{r=1}^t (-1)^{r-1} \binom{\mu}{r} r = 
 \frac{\mu}{\Gamma(2-\mu)} \frac{\Gamma(t-\mu+1)}{\Gamma(t)}.
$$
Plugging these expressions into (\ref{expected_events_def_sib_stop}) we obtain 
\beq
\label{expected_events_def_sib_stop_result}
\mathbb{E} M(t)  = \mathcal{P}_s p_0t + \mathcal{Q}_s\, p_0 \, \frac{\Gamma(t-\mu+1)}{\Gamma(2-\mu) \, \Gamma(t) } , \hspace{1cm} t \geq 1
\eeq
and $\mathbb{E} M(0) =0$.
%
%
%
%
Interestingly, as $t$ goes to infinity
\beq
\label{large_time_def_sib_stop}
\mathbb{E} M(t)  \sim  \mathcal{P}_s p_0 t + \mathcal{Q}_s\, p_0 \, \frac{t^{1-\mu}}{\Gamma(2-\mu)},  
\qquad \mu \in (0,1).
\eeq
%
%
%
\\
Remarkable is the non-defective limit $\mathcal{Q}_s =1$ in which the stopping time is drawn from a Sibuya distribution and
$\mathbb{E} M(t)$ increases as a power of order $1-\mu$. This behavior reflects the fat-tailed feature of the Sibuya stopping time:
despite each realization of $M(t)$ is stopped a.s. 
 all integer moments diverge. 
\\
Fig.~\ref{Ber_defective_Sib_stopped} shows a comparison between Monte Carlo simulations and the analytical result (\ref{expected_events_def_sib_stop_result}).
%
%
\begin{figure}
\centerline{\includegraphics[width=0.7\textwidth]{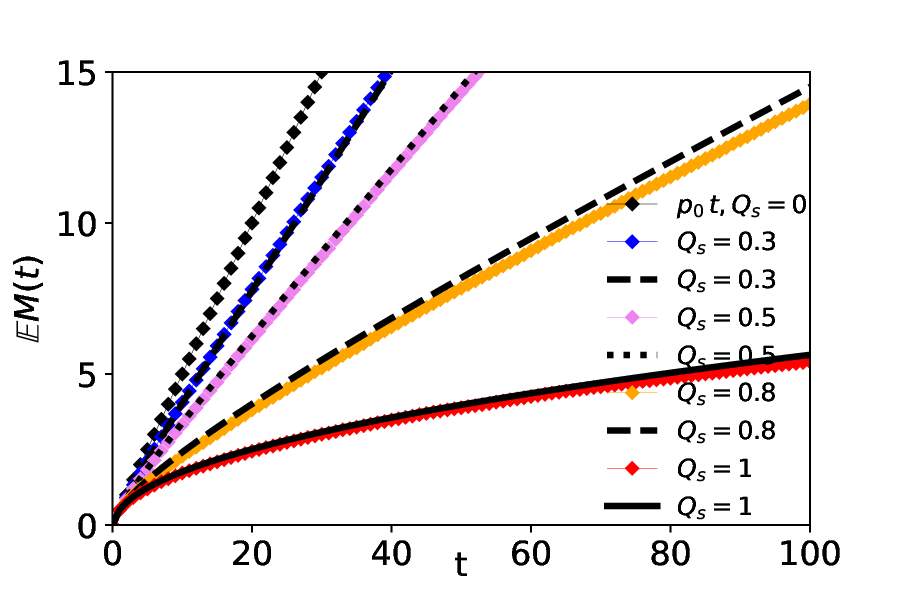}}
\vspace{-3mm}
\caption{Bernoulli process with $p_0=0.5$ and defective Sibuya stopping time ($\mu=0.5$) for several values $\mathcal{Q}_s$. 
Colored curves refer to Monte Carlo
simulations ($1500$ realizations), black curves show the corresponding analytical expressions, see  (\ref{expected_events_def_sib_stop_result}). }
\label{Ber_defective_Sib_stopped}
\end{figure}
\begin{figure}
\centerline{\includegraphics[width=0.7\textwidth]{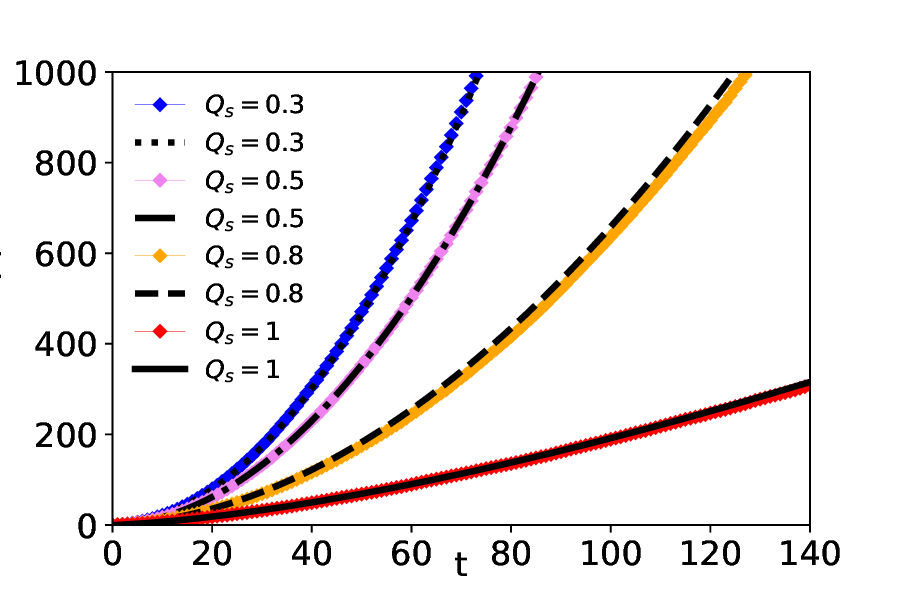}}
\vspace{-3mm}
\caption{
Bernoulli process with $p_0=0.5$ 
and defective Sibuya stopping time ($\mu=0.5$) for several values $\mathcal{Q}_s$. 
Colored curves refer to Monte Carlo
simulations ($2000$ realizations), black curves show the corresponding  analytical expression, see (\ref{second_mo_def_sib_stp_ber}).
}
\label{Ber_defective_Sib_stopped_second_moment}
\end{figure}
Similarly, by using again relation (\ref{higher_moment}) we also derive the second moment.
Using the relations
$$
 (-1)^{t} \binom{\mu-1}{t}\, t(t-1) =  \frac{1}{\Gamma(1-\mu)} \frac{\Gamma(t-\mu+1)}{\Gamma(t-1)} , \hspace{1cm} t \geq 2 
$$
$$\sum_{r=1}^t (-1)^{r-1} \binom{\mu}{r} r(r-1) =   \frac{\mu}{(2-\mu)\Gamma(1-\mu)} \frac{\Gamma(t-\mu+1)}{\Gamma(t-1)} , \hspace{1cm} t \geq 2.  
$$
we get the explicit expression 
%
%
\begin{align}
\label{second_mo_def_sib_stp_ber}
\ds \mathbb{E} M^2(t) &=\ds  \left\{\mathcal{P}_s + \mathcal{Q}_s (-1)^{t} \binom{\mu-1}{t} \right\}[p_0^2t(t-1)+p_0t]+\mathcal{Q}_s\sum_{r=1}^t (-1)^{r-1} \binom{\mu}{r} [p_0^2r(r-1)+p_0r]  \nonumber\\[3ex]
& =\ds  \mathbb{E} M(t) +\mathcal{P}_sp_0^2t(t-1) + \frac{2\mathcal{Q}_sp_0^2 \Theta(t-2)}{(2-\mu)\Gamma(1-\mu)}\frac{\Gamma(t-\mu+1)}{\Gamma(t-1)}  \nonumber \\
& = \ds \mathcal{P}_s(p_0q_0t + p_0^2t^2) + 
\mathcal{Q}_s\, p_0 \, \frac{\Theta(t-1)\Gamma(t-\mu+1)}{\Gamma(2-\mu) \, \Gamma(t) } +\frac{2\mathcal{Q}_sp_0^2 \Theta(t-2)}{(2-\mu)\Gamma(1-\mu)}\frac{\Gamma(t-\mu+1)}{\Gamma(t-1)},
\end{align}
where the discrete Heaviside functions in this expression account for the respective support of the involved terms.
As long as the stopping time is defective, the second moment behaves asymptotically as $t^2$, inheriting this feature from the standard Bernoulli process.
\\
Worthy of mention is the case $\mathcal{Q}_s =1$ where the large time asymptotics is dominated by the $t^{2-\mu}$ term: $$\mathbb{E} M^2(t) \sim 
\frac{2\mathcal{Q}_sp_0^2 }{(2-\mu)\Gamma(1-\mu)} \, t^{2-\mu}.$$
The exponent is in the range $1 < 2-\mu <2$ and can be connected to superdiffusion in a random walk framework \cite{MetzlerKlafter2000}. 

Fig.~\ref{Ber_defective_Sib_stopped_second_moment} shows a comparison between  
the exact expression (\ref{second_mo_def_sib_stp_ber}) (black curves) and Monte Carlo simulations (colored curves) for the same cases as in Fig.~\ref{Ber_defective_Sib_stopped}. The second moment $\mathbb{E} M^2(t)$ increases as $\mathcal{Q}_s$ decreases. 
%
%
%
%
%

\subsection{IAPs and discrete Bernstein functions}
\label{discrete_Bernsetin}
We assume that the PDF $\psi_S(r)$ is a superposition of geometric densities 
$pq^{r-1}$ where we put for our convenience $p=1-e^{-\lambda}$, $\lambda \in (0,\infty)$, namely
\begin{align}
\label{superpos_geo}
\psi_S(r) = \int_0^{\infty} (1-e^{-\lambda})e^{-\lambda(r-1)} \, a({\rm d}\lambda) ,  \hspace{1cm} r\in \mathbb{N}.
\end{align}
Here $a$ is a
probability measure on $\mathbb{R}_+$, which we allow to be defective. Equation  (\ref{superpos_geo}) can be read as a linear combination $\psi_S(r)={\hat a}(r-1)-{\hat a}(r)$ of Laplace transforms
\begin{align}
{\hat a}(t) = \int_0^{\infty}e^{-\lambda t}a({\rm d}\lambda).
\end{align}
The normalization directly follows:
\beq
\label{norma}
\sum_{r=1}^{\infty}\psi_S(r) =\int_0^{\infty} a({\rm d}\lambda) = {\cal Q}_S
\eeq
where ${\cal Q}_S <1 $ if $\Delta T$ is defective
(which corresponds to $a$ being defective). While, in the case ${\cal Q}_S=1$, $a$ is non-defective.

\begin{rmk}
Note that $\psi_S(r) = \mathbb{E} \bigl(\mathbb{P}[Y(1) = r] \bigr)$, where $(Y(\tau))_{\tau \in \mathbb{R}_+}$ is a pure birth process in continuous time, independent of 
%
$\Delta T$ and $N_{II}(t)$, and with random individual birth rate with distribution $a$.
\end{rmk}
Let us briefly highlight the connections of (\ref{superpos_geo}) with discrete versions of Bernstein functions and discrete complete monotonicity \cite{jedidi} 
(consult \cite{Schilling2012,Widder1941} for their continuous counterparts).
To this end, let us introduce the discrete difference operator $D= 1-{\cal T}_{-1}$ with
${\cal T}_{-r}f(t) = f(t-r)$, where $t,r \in \mathbb{Z}$. 
\begin{defin}
\label{CM_discrete}
We call a causal discrete function $f(t) \geq 0$, $t \in \mathbb{N}_0$, discrete completely monotone (DCM) if
\beq
\label{CM}
(-1)^nD^nf(t) \geq 0 , \hspace{1cm} t \geq n , \hspace{1cm} n \in \mathbb{N}_0
\eeq
and  discrete Bernstein (DB) if
\beq
\label{BF}
(-1)^{n-1}D^n f(t) \geq 0 , \hspace{1cm} t \geq n , \hspace{1cm} n \in \mathbb{N}.
\eeq
\end{defin}
\noindent
A basic DCM function is $e^{-\lambda t}$ for $\lambda \geq 0$. Indeed, $(-1)^nD^ne^{-\lambda t}=(e^{\lambda}-1)^n e^{-\lambda t}$ (remaining causal for $t\geq n$). Then, we can infer that any superposition for positive $\lambda$ is DCM, namely
\beq
\label{superposition_measure}
g(t)= \int_0^{\infty} e^{-\lambda t} \nu({\rm d}\lambda) 
\eeq
for any integrable non-negative measure $\nu({\rm d}\lambda)$ (see e.g.\ \cite{jedidi}). This representation of DCM functions can be seen as the discrete version of Bernstein's theorem (also known as Hausdorff–Bernstein–Widder theorem), which justifies (\ref{superpos_geo})
where $\psi_S(r)$ is in the class of DCM functions. This remains true in both the defective and non-defective cases.
It is worthy of mention that the cumulative distribution function
\beq
\label{CDF}
1-\Phi_S^{(0)}(t) = \sum_{r=1}^t \psi_S(r) =  
\int_0^{\infty} (1-e^{-t\lambda}) \, a({\rm d}\lambda) , \hspace{1cm} t \in \mathbb{N}_0
\eeq
is a DB function because so is $(1-e^{-t\lambda})$,
and that  $\Phi_S^{(0)}(t)$ is a DCM function. Moreover, (\ref{CDF}) is in fact the L\'evy--Khintchine representation of the DB function $1-\Phi_S^{(0)}(t)$ where $a$ represents the corresponding L\'evy measure.
\newline
Plugging (\ref{superpos_geo}) into (\ref{Lambda-M}), and recalling  (\ref{GF_Fn0}) leads to
\beq
\label{Lambda-M_limit_lambda-int}
\Lambda_M = 1- {\cal Q}_S + \lim_{m_0\to \infty}  \int_0^{\infty}  e^{\lambda}\, [{\overline \psi}_{II}(e^{-\lambda})]^{m_0+1}a({\rm d}\lambda).
\eeq
Consider now the case $a({\rm d}\lambda) = a(\lambda) {\rm d}\lambda$ where $a(\lambda)$ is a defective or non-defective PDF with respect to $\lambda$.
Note that $e^{\lambda} {\overline \psi_{II}}(e^{-\lambda}) < 1$ for $\lambda >0$ due to the convexity of ${\overline \psi_{II}}(u)$ for $u\in (0,1)$. Now, we rescale $\lambda$ as $m_0^{-\frac{1}{\mu}}\lambda $, $\mu \in (0,1]$,
to get, as $m_0 \to \infty$,
\beq
\label{Lambda-M_limit_lambda-int_rescale}
\Lambda_M \sim 1- {\cal Q}_S + \frac{1}{m_0^{\frac{1}{\mu}}} \int_0^{\infty}  
e^{\lambda m_0^{-\frac{1}{\mu}}}
[{\overline \psi}_{II}(e^{-\lambda m_0^{-\frac{1}{\mu}}})]^{m_0} 
a(\lambda m_0^{-\frac{1}{\mu}}){\rm d}\lambda
\eeq
with the two cases
\beq
\label{two-cases_psi_II}
{\overline \psi}_{II}(e^{-h}) = \ds \left\{ \begin{array}{l} 1-A_1h +o(h) \\ \\ \ds
1-A_{\mu} h^{\mu} +o(h^{\mu}), \hspace{1cm} \mu \in (0,1) 
\end{array}\right.  \hspace{2cm} h \to 0
\eeq
where $A_1,A_{\mu} >0$. The first line refers to the case in which $\psi_{II}(t)$ is light-tailed and has finite mean $A_1=\sum_{r=1}^{\infty}r\psi_{II}(r)$, whereas $\mu<1$ is the case of a fat-tailed $\psi_{II}(t)$ with infinite mean (such as the Sibuya distribution).
We have
\beq
\label{two-cases_psi_II_power}
\ds \lim_{m_0\to \infty}[{\overline \psi}_{II}(e^{-m_0^{-\frac{1}{\mu}} \lambda})]^{m_0} = \left\{ \begin{array}{l} \ds
e^{-A_1\lambda} \\ \\ \ds
e^{-A_{\mu}\lambda^{\mu}}, \hspace{1cm} \mu \in (0,1).
\end{array}\right.  
\eeq
Now we assume in (\ref{superpos_geo}) the following density
\beq
\label{density_a}
a_{\gamma,\zeta}(\lambda)= e^{-\zeta \lambda} \frac{\lambda^{\gamma-1}}{\Gamma(\gamma)} , 
\hspace{1cm} \zeta \geq 1, \hspace{0.5cm} \gamma > 0
\eeq
which is a defective PDF for $\zeta >1$ with ${\cal Q}_S = \zeta^{-\gamma}$ and is a  $\Gamma$-distribution for
$\zeta=1$. In the range $\gamma \in (0,1)$, the density (\ref{density_a}) is weakly singular at $\lambda=0+$ with $a_{\gamma}(\lambda) \sim  \frac{\lambda^{\gamma-1}}{\Gamma(\gamma)}$.
From (\ref{superpos_geo}) we get
\beq
\label{psi_gamma}
\psi_S^{(\gamma,\zeta)}(t)= \frac{1}{(t-1+\zeta)^{\gamma}} - \frac{1}{(t+\zeta)^{\gamma}} , \hspace{1cm}
t \in \mathbb{N}  , 
\hspace{1cm} \zeta \geq 1, \hspace{0.5cm} \gamma > 0
\eeq
that for $\gamma \in (0,1]$
is fat-tailed as $\psi_S^{(\gamma,\zeta)}(t) \sim \gamma t^{-\gamma-1}$, as $t\to \infty$ (i.e.\ with diverging expected absorbing time). For $\gamma > 1$, $\psi_S^{(\gamma,\zeta)}(t)$, exhibits finite mean stopping time in the non-defective case $\zeta=1$. The probability that the process is not stopped up to time $t$ writes
\beq
\label{defective_case}
\Phi_S^{(0),\gamma,\zeta}(t) = 1-\sum_{r=1}^{t}\psi_S^{(\gamma,\zeta)}(r)= 1-\frac{1}{\zeta^{\gamma}}+ \frac{1}{(t+\zeta)^{\gamma}} ,\hspace{1cm} t \in \mathbb{N}_0  , 
\hspace{1cm} \zeta \geq 1, \hspace{0.5cm} \gamma > 0
\eeq
fulfilling the initial condition $\Phi_S^{(0),\gamma,\zeta}(t)\big|_{t=0} =1$ and
$\lim_{t\to \infty} \Phi_S^{(0),\gamma,\zeta}(t) =1-\zeta^{-\gamma}$ remaining positive in the defective case.
\\
Then, the probability that $M(t)$ is never stopped (see (\ref{Lambda})) then reads
\beq
\label{fat-tailedpsi_s}
\Lambda_M = 1-\zeta^{-\gamma} +\lim_{m_0\to\infty} m_0^{-\frac{\gamma}{\mu}} \int_0^{\infty}
e^{-A_{\mu}\lambda^{\mu}} e^{-\lambda(\zeta-1)m_0^{-\frac{1}{\mu}}} \frac{\lambda^{\gamma-1}}{\Gamma(\gamma)} {\rm d}\lambda =
 1-\zeta^{-\gamma} = \Phi_S^{(0),\gamma,\zeta}(\infty)
\eeq
where $e^{-\lambda(\zeta-1)m_0^{-\frac{1}{\mu}}} \to 1$. Hence, the integral converges to a constant
$C = \int_0^{\infty}
e^{-A_{\mu}\lambda^{\mu}} \frac{\lambda^{\gamma-1}}{\Gamma(\gamma)} {\rm d}\lambda$ with
$\Lambda_M \sim 1-\zeta^{-\gamma} +  m_0^{-\frac{\gamma}{\mu}} C$. 

By suitably choosing $\zeta$,  in the general DCM class (\ref{superpos_geo}) with density (\ref{density_a}) we generate examples within each family of arrival processes: type I, IAP and type II for  $\zeta=1$, $\zeta \in (1,\infty)$ and $\zeta=\infty$, respectively. 
%
%
%
%
%
%
%
%
\section{Application to random walks}
\label{RWs}

\subsection{General framework}
\label{gen_fram}
Let $\vec{X}_t$ be the position vector of a discrete time random walk on $\mathbb{Z}^d$ defined as
\beq
\label{random_position}
\vec{X}_t = \sum_{r=1}^{{\mathcal M(t)}} \vec{\xi}_r = \sum_{\ell=1}^{\infty} \vec{\xi}_{\ell} \, \Theta(t-T_{\ell})  , \hspace{1cm} \vec{X}_0 = \vec{0} , \hspace{1cm} t \in \mathbb{N}_0 ,
\eeq
%
where $\vec{\xi}_r$ denote the steps of the walker which we assume to be independent of $\mathcal M(t)$. 
The walker, at $t=0$, starts at the origin. 
Let $\mathcal M(t) = \max(m \in \mathbb{N}_0: T_m \leq t)$ be a 
discrete-time AP 
and let us denote its arrival times as $T_m \in \mathbb{N}$.
%
%
We call $\mathcal M(t)$ the generator process of the walk (\ref{random_position}). Clearly, (\ref{random_position}) is a random sum (see e.g.\ \cite{Pinsky_Karlin}).
We further assume that the steps $\vec{\xi}_r$ and the waiting times $T_m-T_{m-1}$ are independent.
%
The steps $\vec{\xi}_r$ take values in $\mathbb{Z}^d$ and are drawn from the discrete PDF
$\mathbb{P}[\vec{\xi} = \vec{y}] = \mathbb{E} \delta_{\vec{y},\vec{\xi}} = W(\vec{y})$ independent of $\mathcal M(t)$
and are such that their first and second moments are finite.
Per construction (\ref{random_position}) is a random walk time-changed with ${\mathcal M(t)}$, which represents its operational time. If ${\mathcal M(t)}$ is a renewal process, the walk (\ref{random_position}) is a discrete-time version of a Montroll--Weiss type walk \cite{MontrollWeiss1968}.
For the following it is convenient to rewrite (\ref{random_position}) as
\beq
\label{MW_rep}
\vec{X}_t = \sum_{n=0}^{\infty} \Theta(T_n,t,T_{n+1}) \sum_{r=1}^n \vec{\xi}_r,
\eeq
involving the indicator function $\Theta(T_n,t,T_{n+1})$  (see (\ref{Theta-a-b})). 
Further, we denote by $\delta_{\vec{x},\vec{X}_t}$ the indicator function sitting at $\vec{x}$ of the 
position of the walker. The notation
$\delta_{\vec{a},\vec{b}} = \delta_{a_1,b_1}\cdot \ldots \cdot \delta_{a_d,b_d} $ stands for the $d$-dimensional Kronecker symbol, where $a_i$ and $b_i$, $i \in \{1,\dots,d\}$, represent the components of the vectors $\vec{a}$ and $\vec{b}$, respectively. 
Then, considering the representation (\ref{MW_rep}) we have

\beq
\label{indicator}
\delta_{\vec{x},\vec{X}_t} = \sum_{n=0}^{\infty} \delta_{\vec{x}, \sum_{r=1}^n \vec{\xi}_r}\,  \Theta(T_n,t,T_{n+1}) .
\eeq
By taking the expectation of \eqref{indicator} we obtain the propagator, that is the spatial PDF $P(\vec{x},t)$ of the walker, which gives the probability of finding the walker on the lattice point $\vec{x} \in \mathbb{Z}^d$ at time $t$:
\beq
\label{propa}
P(\vec{x},t) 
= \sum_{n=0}^{\infty} \mathbb{E} \Theta(T_n,t,T_{n+1})   \mathbb{E} \delta_{\vec{x}, \sum_{r=1}^n \vec{\xi}_r}.
\eeq
In \eqref{propa} we used the independence between the steps $\vec{\xi}_r$ and the arrival process ${\mathcal M(t)}$. Further, the quantity 
$\mathbb{E} \Theta(T_n,t,T_{n+1})  = \mathbb{P}[{\mathcal M(t)}=n] $, the state probabilities of ${\mathcal M(t)}$,
and $\mathbb{E} \delta_{\vec{x}, \sum_{r=1}^n \vec{\xi}_r}  =  
\mathbb{E} [\delta_{\vec{x},\vec{\xi}} \, \star]^n $ refers to the transition matrix of $n$ independent steps $\vec{\xi}_r$.
The Fourier representation of the spatial PDF is
\begin{eqnarray}
\label{spatial_PDF}
 P(\vec{x},t) &=& \mathbb{E}  \frac{1}{(2\pi)^d} \int_{-\pi}^{\pi} \ldots \int_{-\pi}^{\pi} e^{i\varphi(\vec{x}-\vec{X}_t)\cdot \vec{\varphi}}  {\rm d}\varphi_1 \ldots {\rm d}\varphi_d \notag \\
& = &  \frac{1}{(2\pi)^d} \int_{-\pi}^{\pi} \ldots \int_{-\pi}^{\pi} e^{i\vec{\varphi}\cdot\vec{x}} {\hat P}(\vec{\varphi},t) {\rm d}\varphi_1 \ldots {\rm d}\varphi_d 
\end{eqnarray}
where we used the Fourier representation of the Kronecker symbol $ \delta_{r,s} = \frac{1}{2\pi} 
\int_{-\pi}^{\pi} e^{i\varphi(r-s)}{\rm d}\varphi $, $r,s \in \mathbb{Z}^d$, and wrote $\vec{a}\cdot\vec{b}=\sum_{r=1}^d a_rb_r$ for scalar products. 
The characteristic function of the random walk $ \vec{X}_t$ is
\begin{align}
\label{charact_function}
{\hat P}(\vec{\varphi},t) &= \mathbb{E} e^{-i\vec{\varphi}\cdot \vec{X}_t}  =
\sum_{m=0}^{t} \mathbb{P}[{\mathcal M(t)}=m] [{\hat W}(\vec{\varphi})]^m \notag \\
&= \mathcal{P}_{\mathcal M}({\hat W}(\vec{\varphi}),t) =\mathbb{E} \, {\hat W}(\vec{\varphi})^{\mathcal M(t)} , \hspace{1cm} \varphi_j \in [-\pi, \pi].
\end{align}
%
%
%
Further, ${\hat P}(\vec{\varphi},t)\big|_{\vec{\varphi}=\vec{0}}=1$ gives the spatial normalization of the propagator.
Finally, notice the notation
\beq
\label{transitionPDF-fourier}
{\hat W}(\vec \varphi) = \mathbb{E} e^{-i\vec{\varphi} \cdot \vec{\xi}} = 
\sum_{\vec{x} \in \mathbb{Z}^d} W(\vec{x}) e^{-i\vec{\varphi} \cdot \vec{x}}.
\eeq
Using Wald's identity  we deduce the expected position, denoting with $X_j(t)$ the Cartesian components of
 $\vec{X}_t$, as 
\beq
\label{mean_dis}
\mathbb{E} X_j(t) 
=\mathbb{E}\mathcal M(t) \mathbb{E} \xi_j.
\eeq
From (\ref{charact_function}) the second moment of the position reads
\begin{equation}  
\label{second_moment_spatial}
\ds \mathbb{E} X_j(t)^2 
=  \ds  \mathbb{E}\mathcal M(t)(\mathcal M(t)-1) [\mathbb{E} \xi_j]^2 + \mathbb{E}\mathcal M(t) \mathbb{E} \xi_j^2  = 
\mathbb{E}\mathcal M(t)^2 [\mathbb{E} \xi_j]^2 + \mathbb{E}\mathcal M(t)\mathbb{V}\text{ar}{\xi_j}   
\end{equation}
and therefore,
\beq
\label{Var_X}
\mathbb{V}\text{ar}X_j(t) =\mathbb{V}\text{ar}{\mathcal M(t)} (\mathbb{E} \xi_j)^2 + \mathbb{E}\mathcal M(t) \mathbb{V}\text{ar} {\xi_j}.
\eeq
The GF of the characteristic function (\ref{charact_function}) is
\beq
\label{GF_charfunction}
 {\hat {\overline  P}}(\vec{\varphi},u)  = \sum_{t=0}^{\infty}u^t {\hat P}(\vec{\varphi},t) = \sum_{m=0}^{\infty} {\overline \Phi}_{\cal M}^{(m)}(u)[{\hat W}(\vec{\varphi})]^m 
 \eeq
where the GFs ${\overline \Phi}_{\cal M}^{(m)}(u)$ of the state probabilities $\mathbb{P}[{\mathcal M(t)}=m]$ of the generator process ${\mathcal M(t)}$ come into play.

In the following we consider a graph whose nodes constitutes a subset of $\mathbb Z^d$ and the edges are determined by the distribution of the steps $\vec{\xi}$, thus defining the topology of the graph. 
We assume that all the nodes of the graph have constant degree $\Lambda_d$.

The transition matrix for a step $\vec{x} \to \vec{x}+\vec{y}$, $\vec{x},\vec{y} \in \mathbb{Z}^d$, can then be represented as
\beq
\label{transition_op_representation}
W(\vec{y})  = \mathbb{E} \delta_{\vec{y},\vec{\xi}} = \sum_{r=1}^{\Lambda_d} p_r  \delta_{\vec{y},\vec{a}^{(r)}} 
\eeq
where the walker performs the step $\vec{a}^{(r)}$ with probability $p_r$.
Note that $\delta_{\vec{y},\vec{a}}$ is the matrix describing the deterministic step $\vec{x} \to \vec{x}+\vec{a}$
and (\ref{transition_op_representation}) averages over all possible steps. 
The probabilities $p_r$ can be conceived as the weights of the edges in a weighted
graph and clearly
\begin{align}
\sum_{\vec{y} \in \mathbb{Z}^d} W(\vec{y})  = \sum_{r=1}^{\Lambda_d} p_r=1.
\end{align}
The vectors $\vec{a}^{(r)}$ represent the directed edges of the graph.
This representation allows us to capture a wide range of possibly biased walks on weighted graphs 
(see \cite{RiasMiMedin2020,Riascos-et-al2023,VanMiegem2011} for related models).

In order to focus on the effects induced by the generator process ${\mathcal M(t)}$, we focus on simple walks in which the Fourier transform 
(\ref{transitionPDF-fourier}) takes the form
\beq
\label{fourier_trans}
{\hat W}(\vec{\varphi}) = \sum_{r=1}^{\Lambda_d} p_r e^{-i\vec{\varphi}\cdot \vec{a}^{(r)}}. 
\eeq
For walks on undirected graphs (unbiased walks), all moments of odd order vanish, thus (\ref{fourier_trans}) is real and takes the form  $W(\vec{\varphi}) =  
\sum_{r=1}^{\Lambda_d} p_r \cos{\vec{\varphi}\cdot \vec{a}^{(r)}}$. 
This case includes the well-known unbiased walk on $\mathbb{Z}^d$
with $\Lambda_d = 2d$ and
\beq
\label{unbiased_W}
{\hat W}(\vec{\varphi})=\frac{1}{d}\sum_{r=1}^d\cos{\varphi_r} 
\eeq
(consult for instance \cite{Polya1921} and \cite{grimmett}). 
We study now a few cases of random walks of the form (\ref{random_position}). We use the following classification (cf. Section \ref{classification}):
\begin{defin}
\label{RW_classification}
\end{defin}
\begin{itemize}
\item $\vec{X}_t$ is a type I RW if ${\mathcal M(t)}$ is an AP of type I; 
\item $\vec{X}_t$ is of type II RW if ${\mathcal M(t)}$ is an AP of type II; 
\item $\vec{X}_t$ is an intermediate random walk (IRW) if ${\mathcal M(t)}$ is an IAP. 
\end{itemize}
%
%
%
%
\begin{rmk}
With reference to Remark \ref{quattrouno}, we observe that the corresponding random walks can still be classified according to Definition \ref{RW_classification}. 
\end{rmk}

\subsection{Intermediate random walks: defective Bernoulli stops the random walker}
\label{RW-transient}

The aim of the present section is to analyze an IRW (\ref{random_position}) {with generator process $M(t)$} is the IAP defined by (\ref{conditional-process}).
Here, we assume that 
%
%
the stopping time $\Delta T$ is distributed as the inter-arrival times of a DBP (see section \ref{DBP}). Recall its waiting time PDF is $\psi_S(t)=\mathcal{Q}_Spq^{t-1}$, $t \in \mathbb{N}$ (and $\Lambda_M=1-\mathcal{Q}_S$).
This IRW is a discrete-time version of a Montroll--Weiss walk (that is, a RW time-changed with a type II RP $N_{II}(t)$) which is stopped at the first event of an independent DBP. 
Using (\ref{state_pol_IAP}), the GF (\ref{GF_charfunction}) writes 
\begin{align}
\label{trans_mat_GF_geostop}
\ds {\hat {\overline  P}}(\vec{\varphi},u) & =\ds {\overline \Pi}_{\text{\textit{IAP}}}({\hat W}(\vec{\varphi}),u) \notag \\
& =  \ds \Lambda_M \frac{1-{\overline \psi}_{II}(u)}{(1-u)[1-{\hat W}(\vec{\varphi}){\overline \psi}_{II}(u)]} +
\frac{\mathcal{Q}_S}{q(1-u)}\left(\frac{1-{\overline \psi}_{II}(qu)}{1- {\hat W}(\vec{\varphi}){\overline \psi}_{II}(qu)} 
- p \right).           
\end{align}
For $\Lambda_M=1$ we have the discrete-time counterpart of the Montroll--Weiss walk.
If $\mathcal{Q}_S=1$
it is a type I RW with geometric absorbing time.
Note that ${\hat {\overline  P}}(\vec{\varphi},u)|_{\vec{\varphi}=\vec{0}} = (1-u)^{-1}$ and 
${\hat {\overline  P}}(\vec{\varphi},u)|_{u=0}=1$
are consistent with the normalization of the propagator and the initial condition $\vec{X}_0=\vec{0}$ a.s., respectively.

\subsubsection{Case: ${N_{II}(t)}$ is a Bernoulli process}
\label{walk-NII-Ber}
Let $N_{II}(t)$ be a Bernoulli process with $\psi_{II}(t)$ with parameter $p_0$.
Then, by using (\ref{GF_Dber_Ber_statepoly}), the characteristic function \eqref{charact_function} takes the explicit form
\beq
\label{propa_charact}
{\hat P}(\vec{\varphi},t) = \Pi_{IAP}({\hat W}(\vec{\varphi}),t) =
\Lambda_M {\cal A}^t + \frac{\mathcal{Q}_S}{1-q{\cal A}}\left(p{\cal A}+(1-{\cal A})q^t{\cal A}^t\right) , \hspace{1cm} {\cal A}(\vec{\varphi}) =q_0+p_0{\hat W}(\vec{\varphi}).
\eeq
In order to explore the diffusive features of the RW, we investigate the first two moments of the position $\vec{X}_t$ of the walker. Recalling (\ref{IAM_1st_moment}), the expected value of the $j$-th component of the position of the walker is 
\beq
\label{first-moment-in-time}
\ds \mathbb{E} X_j(t) =  \mathbb{E} M(t) \, \mathbb{E} \xi_j = 
 \left[\Lambda_M p_0t+\frac{\mathcal{Q}_Sp_0}{p}(1-q^t) \right]\,  \mathbb{E} \xi_j.   
\eeq
For large values of $t$, the drift of this walk is dominated by the linear contribution of the Bernoulli process $N_{II}(t)$.
From (\ref{second_moment_spatial}), together with (\ref{IAM_1st_moment}) and (\ref{M_second}), the second moment of the $j$-th component of the position is

\begin{align}
\label{MST-in-time}
\ds \mathbb{E} X_j^2(t) = {} & \ds 
\mathbb{E} M^2(t) (\mathbb{E}\xi_j)^2 + \mathbb{E}M(t) \mathbb{V}\text{ar} \, {\xi_j} \notag  \\ 
= {} &  \ds \left(\Lambda_M(p_0^2t^2 +p_0q_0t) + \mathcal{Q}_S \biggl[\frac{2p_0^2q}{p^2}(1-q^t -ptq^{t-1}) + \frac{p_0}{p}(1-q^t)\biggr] \right) (\mathbb{E}\xi_j)^2 \notag \\  &  + \ds  \left(\Lambda_M p_0t+\frac{\mathcal{Q}_Sp_0}{p}(1-q^t)\right) \mathbb{V}\text{ar} \, {\xi_j}. 
\end{align}

(i) \textbf{Unbiased walk.}
We observe the following behavior. If the walk is unbiased (i.e. if $\mathbb{E}\xi_j=0$ $\forall j=1,\ldots, d$), the mean square displacement (MSD) and the variance grow linearly for large times:
\beq
\label{unbiased}
\mathbb{E} \sum_{j=1}^d X_j^2(t) \sim K \, t ,\qquad  K = \Lambda_M p_0  \sum_{j=1}^d \mathbb{E} \xi_j^2, \qquad t \to \infty,
\eeq
corresponding to a  diffusive behavior with  diffusion coefficient $K$ determined by the unbounded sample paths of $N_{II}(t)$. This implies that the asymptotics \eqref{unbiased} does not depend on $p$. However, the  diffusion coefficient $K$ decreases as $\mathcal{Q}_S$ approaches one, and is  maximal if $\mathcal{Q}_S=0$.
\begin{rmk}
    In the simple  unbiased walk on $\mathbb{Z}^d$ with next neighbor steps
    along the  edges of the d-dimensional
hypercubic lattice \cite{grimmett},
    occurring with equal probability $1/(2d)$ we have
$\mathbb{E} \xi_j  = 0$ and $\mathbb{E} \xi_j^2 = 1/d $ ($j=1,\ldots, d$).
The diffusion coefficient defined in (\ref{unbiased}) then is $K=\Lambda_M p_0$.
Note that different $\Delta T$ and $N_{II}(t)$, and hence different values of $\Lambda_M$ and $p_0$, are able to reproduce the same diffusion coefficient $K$.
\end{rmk}
(ii) \textbf{Biased walk.}
The walk is biased if $\mathbb{E} \xi_j \neq 0$ for at least one component $j$. From \eqref{MST-in-time} we observe a ballistic superdiffusive behavior
\beq
\label{ballisstic}
\mathbb{E} X_{j}^2(t)\sim  \Lambda_M p_0^2 (\mathbb{E} \xi_{j})^2 \, t^2  , \qquad t \to \infty.
\eeq
The dominating term of the asymptotic variance of the position is quadratic and comes from the variance of the generator $M(t)$. 

We notice that the behaviors of cases (i) and (ii) are well-known in the context of continuous-time random walks \cite{Shlesinger1974,ScherMontroll1975}.
The above theory is valid for every Bravais lattice \cite{martin_perk2019} of which $\mathbb{Z}^d$ is the most prominent example. A further example is the triangular lattice, which we discuss in the two-dimensional case in the next section.
\subsubsection{IRWs on the triangular lattice}
We consider biased and unbiased versions of the IRW (\ref{random_position}) on an infinite triangular lattice in the case in which the generator process 
${M(t)}$ is the IAP defined by (\ref{conditional-process}).
 
In the biased RW we allow only four steps towards neighbor lattice points
$\vec{a}^{(r)} = (\cos\frac{r\pi}{3}, \sin \frac{r\pi}{3})$, $r \in\{0,1,2,3\}$, each with equal probability $p_r=1/4$, see Fig. \ref{biased_trinagular}(a).
In the unbiased RW steps to all the $6$ closest neighbor lattice points $\vec{a}^{(r)}$, $r \in\{0,\ldots,5\}$, occur with equal probability $p_r=1/6$, see Fig. \ref{biased_trinagular}(b).
\begin{figure}
\centerline{\includegraphics[width=0.7\textwidth]{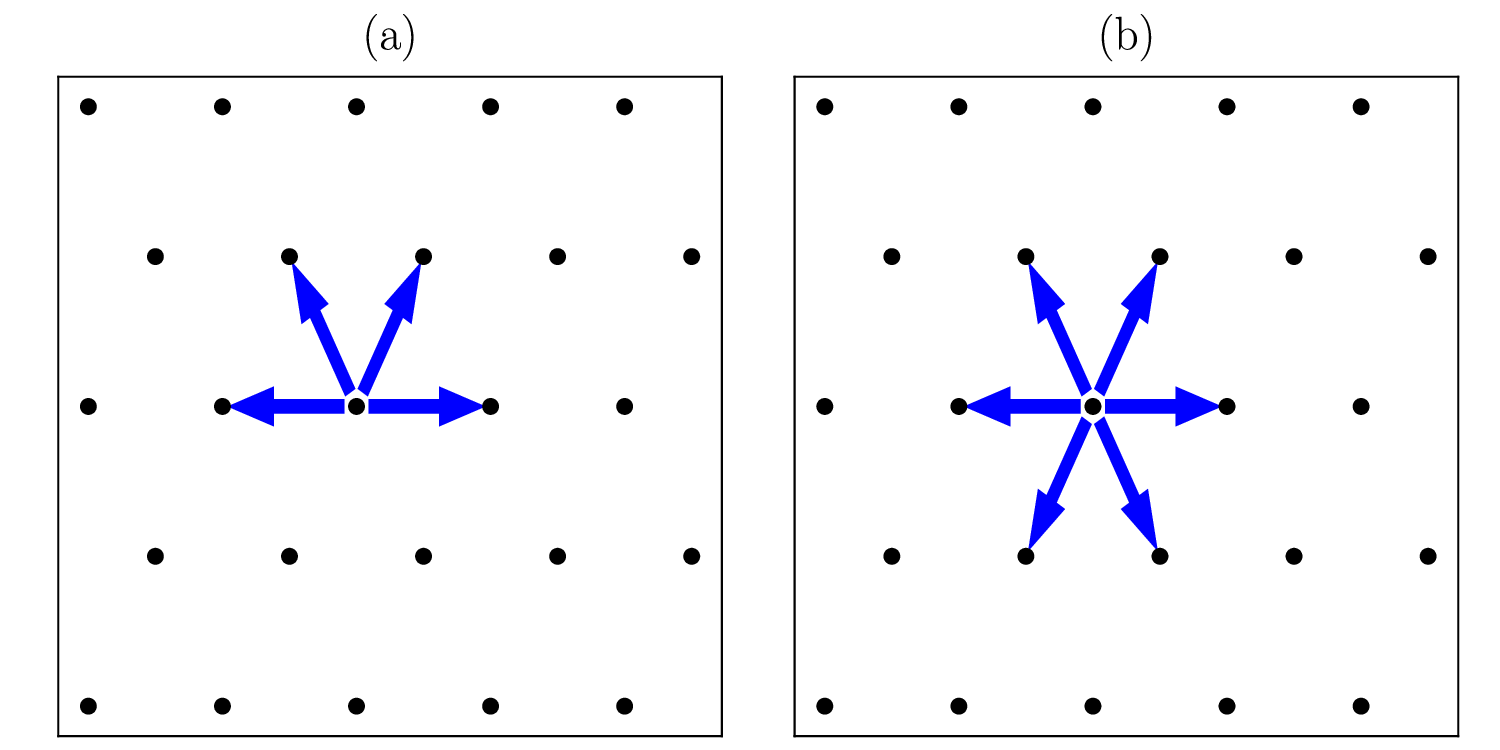} 
}
\caption{ Blue arrows represent the allowed and equally probable steps in a biased walk (panel (a)) and in an unbiased walk (panel (b)) on the triangular lattice with lattice constant set equal to unity.}
\label{biased_trinagular}
\end{figure}
(i) \textbf{Unbiased case.}
We have that  $\mathbb{E} \xi_j=0$ and $\mathbb{E}(\xi_1^2+\xi_2^2) = 1$  thus the MSD reads
\beq
\label{MSD_unbiased-triangular}
\sum_{j=1}^2 \mathbb{E} X_j^2(t) = \mathbb{E} M(t). 
\eeq
In the special case of a DBP stopped by a  Bernoulli r.v., for large values of $t$ a diffusion behavior emerges. Namely,
$\sum_{j=1}^2 \mathbb{E} X_j^2(t) \sim \Lambda_M p_0 t$, (see (\ref{IAM_1st_moment}) and Fig.~\ref{Expected_M})). 

\medskip

(ii) \textbf{Biased case.}
We have  $\mathbb{E} \xi_1 = 0$, $\mathbb{E} \xi_2 = \frac{\sqrt{3}}{4}$,  $\mathbb{E}(\xi_1^2 +\xi_2^2) = 1$ 
 giving the MSD:
\beq
\label{MSD_biased-triangular}
\sum_{j=1}^2 \mathbb{E} X_j^2(t) = \frac{3}{16} \mathbb{E} M^2(t) +\frac{13}{16} \mathbb{E} M(t).
\eeq
Again in the DBP stopped by a  Bernoulli r.v. we observe a ballistic law that is 
\begin{equation}    
\sum_{j=1}^2 \mathbb{E} X_j^2(t)  \sim  \frac{3}{16} \Lambda_M p_0^2t^2, \qquad t \to \infty. 
\end{equation}
Of some interest is also the expected position of the walker at time $t$. Since $\mathbb{E} \xi_1=0$ (and thus $\mathbb{E}X_1(t) =0$) only the second component is non-zero. In particular,
\beq
\label{X2-direction}
\mathbb{E}X_2(t) = \frac{\sqrt{3}}{4} \mathbb{E} M(t)\sim \frac{\sqrt{3}}{4} \Lambda_Mp_0 t, \qquad t\to \infty, 
\eeq
showing a linear drift.
We plot the exact expression of the MSD (\ref{MSD_biased-triangular}) in Fig. \ref{MSD_biased_fig} for the same values of $\mathcal{Q}_S$
as in Figs.~\ref{Expected_M},~\ref{Var_plot}. 
In the type I limit of the walk ($\mathcal{Q}_S=1$) the MSD has a finite asymptotic value (blue curve). One can clearly see that
the smaller $\mathcal{Q}_S$ the stronger the MSD grows. The black line corresponds to the  type II limit of the walk. 
\begin{figure}
\centerline{\includegraphics[width=0.7\textwidth]{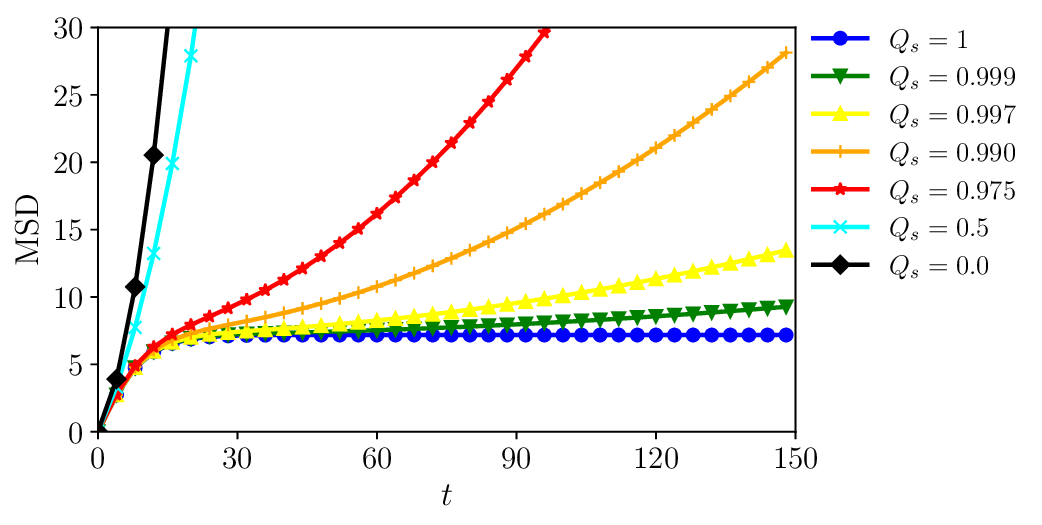}}
\vspace{-4mm}
\caption{MSD of the RW in the biased case from (\ref{MSD_biased-triangular}). The  parameters are chosen as in Fig.~\ref{Var_plot}.}
\label{MSD_biased_fig}
\end{figure}
\section*{Acknowledgments}
 G.D. and F.P. have been partially supported by the MIUR-PRIN 2022 project \lq\lq Non-Markovian dynamics and non-local equations\rq \rq, no. 202277N5H9 and by the GNAMPA group of INdAM. We thank two anonymous reviewers for their valuable comments helping us to 
 substantially improve the presentation.

\hfill

\section*{APPENDIX}

\begin{appendix}
\section{Further results on random walks: type I case}

\label{type_I_limit}
Let us investigate now the type I limit of $M(t)$, that is the stopping time $\Delta T$ is non-defective
For simplicity we consider the case in which $\Delta T$ is geometrically distributed
%
%
%
%
and $N_{II}(t)$ is an arbitrary type II RP as considered in Section \ref{geometric_stop}.
This case generates a RW of type I. To this end we evaluate the characteristic function (\ref{charact_function}) for $t\to \infty$ 
(see (\ref{GF_Pm}): 
\beq
\label{charfunction_geostop}
{\hat P}(\vec{\varphi},\infty) = \Pi({\hat W}(\vec{\varphi}),\infty) =\lim_{u\to 1-}(1-u) {\hat {\overline  P}}(\vec{\varphi},u)  = - \frac{p}{q}+ \frac{1-{\overline \psi}_{II}(q)}{q[1- {\hat W}(\vec{\varphi}) \, {\overline \psi}_{II}(q)]} .
\eeq
The quantity
\beq\label{DMW}
{\hat {\overline P}}_q(\vec{\varphi},u) := \frac{1-{\overline \psi}_{II}(qu)}{(1-u)[1-{\hat W}(\vec{\varphi}){\overline \psi}_{II}(qu)]} 
\eeq
refers to a Montroll--Weiss-type walk (\ref{random_position}) with as a generator the renewal process ${\cal R}_q(t)$ with defective waiting time density $ \psi_q(t) = \psi_{II}(t) q^t$ (see
the end of Section \ref{geometric_stop}). 
Then, the propagator (\ref{trans_mat_GF_geostop}) (with $\Lambda_M=0$) has the representation
\beq
\label{space-time}
P(\vec{x},t)= \frac{1}{q} P_q(\vec{x},t) -\frac{p}{q} \delta_{\vec{x},\vec{0}}
\eeq
with the initial condition $P(\vec{x},t)|_{t=0} = P_{q}(\vec{x},t)|_{t=0} = \delta_{\vec{x},\vec{0}}$. 

\noindent Note that for $p \in (0,1)$ the type I nature of the walk implies the existence of the stationary PDF, which reads
\beq
\label{space-inftime}
P(\vec{x},\infty) = \frac{1}{q} P_q(\vec{x},\infty) -\frac{p}{q} \delta_{\vec{x},\vec{0}} ,  \hspace{1cm} \vec{x} \in \mathbb{Z}^d,
\eeq
where 
\beq
\label{limiting}
 P_q(\vec{x},\infty) = \frac{1-{\overline \psi}_{II}(q)}{(2\pi)^d} \int_{-\pi}^{\pi} \dots \int_{-\pi}^{\pi} \frac{e^{i\vec{\varphi} \cdot \vec{x}}}{[1-{\hat W}(\vec{\varphi}){\overline \psi}_{II}(q)]}  {\rm d}\varphi_1\ldots{\rm d}\varphi_d ,  \hspace{1cm} \vec{x} \in \mathbb{Z}^d.
 \eeq 
This limit propagator describes a `non-equilibrium steady state' (NESS) as known in the context of stochastic resetting, see \cite{Barkai-etal-2023,Eule-et-al2016}.

\subsection{Continuous-space scaling limit to a universal NESS}
\label{stationary}
Here, we consider a scaling limit for $p \to 0+$ where the infinite time propagator 
(\ref{space-inftime}) boils down to (\ref{limiting}).
To do that, we introduce the scaling parameter (see (\ref{limit_M_Ber}))
\beq
\label{limit_param}
 \lambda \sim \mathbb{E} M(\infty) = \frac{{\overline \psi}_{II}(q)}{q[1-{\overline \psi}_{II}(q)]} \sim \frac{1}{1-{\overline \psi}_{II}(q)} \sim
 \frac{1}{C_{\mu}p^{\mu}}  \to \infty  ,\qquad p\to 0+,
\eeq
where we have used the asymptotic expansion ${\overline \psi}_{II}(1-p) \sim 1- C_{\mu} p^{\mu} +o(p^{\mu})$, $C_{\mu} >0$, $\mu \in (0,1]$ (see also \eqref{two-cases_psi_II}). Clearly, (\ref{space-inftime}) and (\ref{limiting}) converge to the same limiting propagator which we denote by $P_c(\vec{x},\infty)$. The latter is
related to the asymptotic behavior of the type II  auxiliary renewal process $\mathcal{R}_q(t)$ whose interarrival time PDF becomes non-defective for $q\to 1-$ (see (\ref{typeI_renewal_state}--\ref{renewal_eq_R}).
\\
We will see in the following that the emerging NESS is universal, in the sense that it is the same for any type II renewal process $N_{II}(t)$: it depends only on the 
probabilistic properties of the steps captured by ${\hat W}(\vec{\varphi})$.

In the following sections, we derive the NESS propagator for special classes of walks.
\paragraph{Light-tailed step distribution.}
Here we investigate the class of walks  such that the first and second step moments $A_j = \mathbb{E} \xi_j$ and $B_j =\mathbb{E} \xi_j^2$ are finite. For this class we have the expansion
\beq
\label{Wphi}
{\hat W}(\vec{\varphi}) = 1 -i\sum_{j=1}^d A_j\varphi_j -\frac{1}{2} \sum_{j=1}^d\varphi_j^2B_j + \ldots, \qquad |\vec{\varphi}| \to 0.
\eeq
In particular, we focus on the two cases of  biased and unbiased walks.
\bigskip

(i) \textbf{Biased case ($A_j \neq 0$ for at least one $j$).}

\noindent The Fourier transform of (\ref{limiting}) then writes 
\beq
\label{axiliary_repre}
\frac{1}{1-\frac{{\overline \psi}_{II}(q)}{1-{\overline \psi}_{II}(q)}[{\hat W}(\vec{\varphi}) -1]} = \frac{1}{1+\lambda[1-\hat W(\vec{\varphi})]}
= \int_0^{\infty} e^{-\tau[1+\lambda[1-\hat W(\vec{\varphi})]} {\rm d}\tau \overset{\vec{\varphi} \to \vec0}{\longrightarrow} \int_0^{\infty} e^{-\tau(1+i\vec{k}\cdot{\vec{A}})}{\rm d}\tau
\eeq
where $k_j = \lambda \varphi_j \in [-\pi\lambda, \pi\lambda]$ and we denote the  components of the rescaled position vector by $y_j = x_j/\lambda$,\, (that is $\vec{Y} \in \lambda^{-d} \mathbb{Z}^d \to \mathbb{R}^d$ as $\lambda \to \infty$, and $\lambda^{-1}$ is the rescaled lattice constant). In 
(\ref{axiliary_repre}) we use 
\begin{equation}
\lambda\left[1-{\hat W}\left(\frac{\vec{k}}{\lambda}\right)\right] = i\sum_j^dA_jk_j+ \frac{1}{2\lambda}\sum_j^d B_j k_j^2 + o(\lambda^{-1}) \ \overset{\lambda \to \infty}{\longrightarrow} \ 
i\sum_j^dA_jk_j = i\vec{k}\cdot \vec{A}.
\end{equation}
Note that $ {\hat W}(\vec{k}/\lambda) = \sum_{r=1}^{\Lambda_d}p_r e^{-i\vec{k}\cdot \vec{a}_r/\lambda} =\mathbb{E} e^{-i\vec{k}\cdot\vec{a}'}$ contains the rescaled steps $\vec{a}_r'=\vec{a}_r/\lambda$.
The continuous space limiting propagator yields
\beq
\label{limit-large-lamda}
\begin{array}{clr}
 \ds P_c(\vec{y},\infty) & = \ds  \lim_{\lambda  \to \infty} \lambda^d P_q(\vec{x},\infty) = \ds  \lim_{\lambda  \to \infty} \frac{1}{(2\pi)^d}\int_{-\lambda \pi}^{\lambda \pi} \ldots \int_{-\lambda \pi}^{\lambda \pi} \frac{e^{i\vec{k}/\lambda \cdot \vec{x}}}{1+\lambda(1-{\hat W}(\vec{k}/\lambda))} {\rm d}k_1\ldots{\rm d}k_d  &  \\ \\ 
 & = \ds \int_0^{\infty}{\rm d}\tau \,  e^{-\tau} \frac{1}{(2\pi)^d}\int_{-\infty}^{\infty}{\rm d}k_1\ldots 
 \int_{-\infty}^{\infty}{\rm d}k_d e^{i\vec{k}\cdot(\vec{y}-\tau\vec{A})} & \\ \\ & = \ds   
 \int_0^{\infty}{\rm d}\tau \,  e^{-\tau} \delta^d(\vec{y}-\tau\vec{A}). &
\end{array}
\eeq
For $d=1$ we can evaluate this case explicitly, obtaining
\beq
\label{fin-result}
P_{c}(y,\infty) \sim \frac{1}{|A|} \Theta\left(\frac{y}{A}\right) e^{-\frac{y}{A}}  
\eeq
which is a one-sided exponential PDF
and is non-null only in the sense on the direction of the bias, i.e. for $\text{sgn}(y)=\text{sgn}(A)$. 
We plot (\ref{fin-result}) for several values of $A$ in Fig. \ref{limitdis-biased}. Note that $A= \mathbb{E}X_c(\infty)$ is the expected final
position of the walk. Further, as $A$ increases, the distribution (\ref{fin-result}) becomes more spread-out.
\begin{figure}
\centerline{\includegraphics[width=0.7\textwidth]{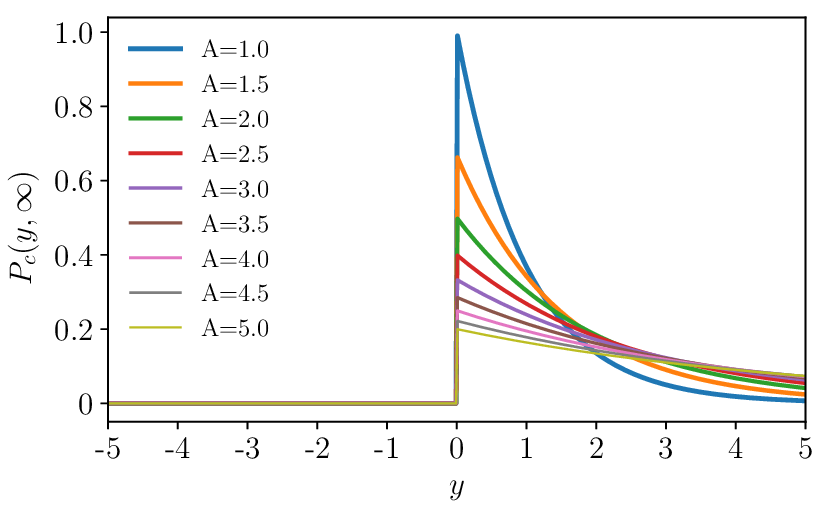}}
\vspace{-3mm}
\caption{ Limiting propagator $P_c(y,\infty)$  of the biased walk from (\ref{fin-result}) for different values of $A$.}
\label{limitdis-biased}
\end{figure}

\bigskip

(ii) \textbf{Unbiased case ($A_j=0$ for  $j=1,\ldots, d$, and $B_j > 0$ for at least one $j$).}

\noindent Differently from the biased case, we rescale the spatial coordinates as $y_j= x_j \lambda^{-1/2}$ with
$\vec{Y} \in \lambda^{-d/2} \mathbb{Z}^d \to \mathbb{R}^d$, as $\lambda \to \infty$, and $k_j=\sqrt{\lambda}\varphi_j \in [-\pi\sqrt{\lambda}, \pi \sqrt{\lambda}]$.
The continuous space limit propagator, considering that $\lambda(1-\hat{W}(\vec{k}/\sqrt{\lambda})) \to  \sum_{j=1}^d\frac{B_j}{2} k_j^2$, reads
\begin{eqnarray} 
\label{unbiased_limit}
 \ds P_{c}(\vec{y},\infty) &=& \ds \lim_{\lambda \to \infty} \lambda^{\frac{d}{2}} P_q(\vec{x},\infty) \notag \\
 & = & \ds \frac{1}{(2\pi)^d}  \int_0^{\infty} {\rm d}\tau \, e^{-\tau} 
\int_{-\infty}^{\infty}{\rm d}k_1\ldots \int_{-\infty}^{\infty}{\rm d}k_d \, e^{i\vec{k}\cdot\vec{y} -\frac{\tau}{2}\sum_{j=1}^d B_j k_j^2}.
\end{eqnarray}
Note that each of the above integrals take a 
Gaussian form:
\beq
\label{gaussin_int}
\ds \frac{1}{2\pi}\int_{-\infty}^{\infty} e^{iky}
 e^{-\frac{\tau B k^2}{2}}{\rm d k} = \ds \frac{e^{\frac{-y^2}{2B\tau}}}{\sqrt{2\pi B\tau}} .
\eeq
Therefore,
\beq
\label{well_scales_propa}
 P_{c}(\vec{y},\infty)  = \ds  \int_0^{\infty} 
\frac{  e^{ -\frac{1}{2\tau}\sum_{j=1}^d \frac{y_j^2}{B_j} } }{ \sqrt{(2\pi\tau)^d B}} \, e^{-\tau} \, {\rm d}\tau , \hspace{1cm} B = \prod_{j=1}^d B_j ,
\eeq
i.e. the unbiased limiting propagator turns out to be a superposition of symmetric Gaussians where $\mathbb{E} Y_j^2 =B_j$ is retained.
For $d=1$ we can evaluate this case explicitly: 
\beq
\label{fin_evaluation}
P_{c}(y,\infty) =\frac{1}{2\pi} \int_{-\infty}^{\infty} \frac{e^{iky}}{\frac{B}{2}k^2+1}{\rm d}k = 
\frac{1}{2\pi} \int_{-\infty}^{\infty} \frac{e^{iky}}{\frac{B}{2}(k-\kappa_1)(k-\kappa_2)}{\rm d}k
\eeq
with complex conjugate zeros $\kappa_{1,2}=\pm i\sqrt{\frac{2}{B}}$. Closing the integration path for $y>0$ by an infinite half circle in the upper and for $y<0$ in the lower complex plane and applying the residue theorem yields
\beq
\label{fin_evaluation_residua}
P_{c}(y,\infty) = \Theta(y) \frac{2i}{B} \frac{e^{i\kappa_1y}}{\kappa_1-\kappa_2} -\Theta(-y)\frac{2i}{B} 
\frac{e^{i\kappa_2y}}{\kappa_2-\kappa_1} = \frac{e^{-|y|\sqrt{\frac{2}{B}}}}{\sqrt{2B}}
\eeq
which is the PDF of a Laplace random variable with location parameter zero and scale parameter $\sqrt{B/2}$. Note that $B$ is the variance of the random walk with limiting propagator $P_c(y, \infty)$. 
In Fig.~\ref{limitdis-unbiased} we plot the propagator $P_c(y, \infty)$ for several values of $B$. 
\begin{figure}
\centerline{\includegraphics[width=0.7\textwidth]{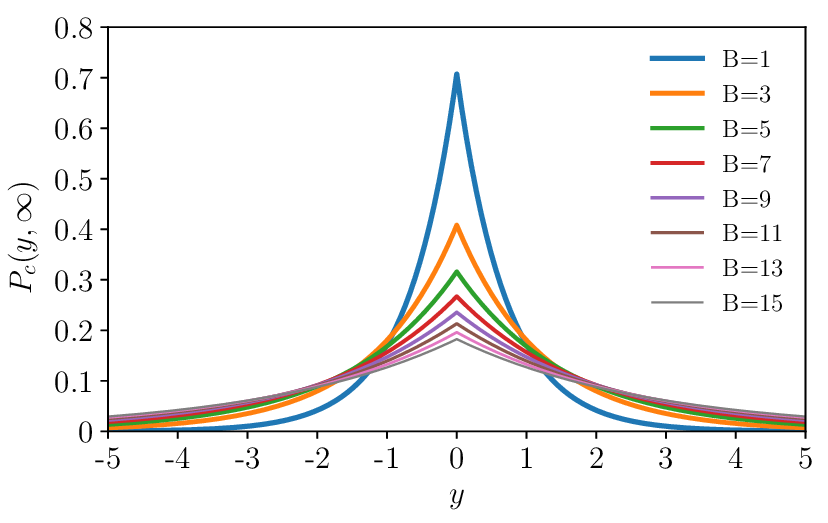}}
\vspace{-3mm}
\caption{Limiting propagator $P_c(y,\infty)$ from (\ref{fin_evaluation_residua}) for the unbiased case for different values of $B$.}
\label{limitdis-unbiased}
\end{figure}
\paragraph{Power-law distributed steps.}
We briefly explore a few cases of walks with power-law distributed steps with infinite mean or variance.
We discriminate between biased and unbiased walks in $d=1$.

\bigskip

{\it (a) Sibuya walk}

\medskip
\noindent Consider a strictly increasing walk with Sibuya distributed steps on $\mathbb{N}$ (\lq\lq Sibuya walk\rq\rq, see e.g.\ \cite{MiPoRia_fractal_fract2020}, Section 4). 
We have then, instead of (\ref{Wphi}), the characteristic function of the steps
\beq
\label{Sibuya_steps}
{\hat W}_{\alpha}(\varphi) = 1-[1-e^{-i\varphi}]^{\alpha} \!\!, \qquad \alpha \in (0,1).
\eeq
We introduce the scaling
$k=\lambda^{\frac{1}{\alpha}}\varphi$. The rescaled walk $Y= \lambda^{-\frac{1}{\alpha}} X \in \lambda^{-\frac{1}{\alpha}}\mathbb{N}$ exhibits the essential feature,
\begin{align}
\lambda[1- {\hat W}_{\alpha}(\lambda^{-\frac{1}{\alpha}} k )] \to (ik)^{\alpha} \! \!, \qquad \lambda \to \infty.
\end{align} 
It turns out that the limiting propagator is
\beq
\label{limit_biased_levy}
\ds P_c(y,\infty)  = \ds \lim_{\lambda \to \infty}\lambda^{\frac{1}{\alpha}}P_q(x,\infty) =
\int_0^{\infty}{\rm d}\tau e^{-\tau} \frac{1}{2\pi}\int_{-\infty}^{\infty}e^{iky} e^{-\tau(ik)^{\alpha}}{\rm d}k. 
\eeq
In formula \eqref{limit_biased_levy}, the inner integral corresponds to the
one-sided stable probability density function ${\cal L}_{\alpha}(y)= 
\frac{1}{2\pi}\int_{-\infty}^{\infty} e^{iky} e^{-(ik)^{\alpha}}{\rm d}k$, $\alpha \in (0,1)$. 
This PDF is causal and has a $y^{-\alpha-1}$ fat-tailed decay for $y\to +\infty$, which entails an infinite mean.
Consult \cite{MiRiaFract-calc2020,WangBarkai2024} for related models of random walks with bias.

\bigskip

{\it (b) One-sided power-law steps with finite mean}

\medskip

\noindent If the strictly increasing walk has finite mean and infinite variance, then the characteristic function of the steps takes the form \cite{GSD-Squirrel-walk2023} 
\beq
\label{biased_diverging_var}
{\hat W}(\varphi) = 1 - A i \varphi + B_{\mu}(i\varphi)^{\mu} +o(|\varphi|^{\mu}), 
\qquad \mu \in (1,2)
\eeq
with $A,B_{\mu} >0$. Due to the lowest order $i\varphi$ in this expansion, the limiting distribution coincides with (\ref{fin-result}).
\paragraph{A unified description of the one-dimensional case: L\'evy steps.}

\medskip

\noindent Now we consider the canonical form of a stable density $\mathcal{L}_\alpha(y)$:
\beq
\label{L_alpha_kanonic}
 {\cal L}_{\alpha}(y) =\frac{1}{2\pi} \int_{-\infty}^{\infty} e^{iky} e^{ - (i^{\theta} k)^{\alpha}} {\rm d}k   , \hspace{1cm} \alpha \in (0,2],
\eeq
containing a further real parameter $\theta$, 
and where 
$ (i^{\theta} k)^{\alpha} = |k|^{\alpha} e^{\frac{\pi i}{2} \text{sgn}(k) \alpha\theta} $. For large $|y|$  and $\alpha \in (0,2)$,  ${\cal L}_{\alpha}(y)$
scales as $|y|^{-\alpha-1}$.
We consider here a random walk whose characteristic function is
\beq
\label{Levy_steps}
{\hat W}_\alpha(\varphi) = 1 -  (i^{\theta} \varphi)^{\alpha} +o(|\varphi|^{\alpha})
\eeq
with  index $\alpha \in (0,2]$. One can show that the previously discussed cases 
are contained in (\ref{Levy_steps}). Now, by the re scaling $k=\lambda^{\frac{1}{\alpha}}\varphi$ 
we obtain 
the Fourier transform of the NESS propagator ${\hat P}_c(k,\infty)=(1+ (i^{\theta} k)^{\alpha})^{-1} $
and hence
\beq
\label{NESS_LEVY}
P_c(y,\infty) = \int_0^{\infty}{\rm d}\tau \, e^{-\tau}  \frac{1}{2\pi} \int_{-\infty}^{\infty}e^{iky} 
e^{-\tau(i^{\theta} k)^{\alpha}}  {\rm d}k = 
\int_0^{\infty}{\rm d}\tau \, e^{-\tau} \tau^{-\frac{1}{\alpha}} \frac{1}{2\pi} \int_{-\infty}^{\infty}e^{i\kappa X/\tau^{\frac{1}{\alpha}}} 
e^{-(i^{\theta} \kappa)^{\alpha}}  {\rm d}\kappa.
\eeq
Formula (\ref{Levy_steps}) contains as a special case the class of symmetric L\'evy flights for which the steps are such that ${\hat W}_{\alpha}(\varphi) = 1 - |\varphi|^{\alpha} +o(|\varphi|^{\alpha})$, that is symmetric $\alpha$-stable random variables with PDF ${\cal L}_{\alpha}(y) = \frac{1}{2\pi}\int_{-\infty}^{\infty} e^{iky} e^{-|k|^{\alpha}}{\rm d}k$.
Clearly, our discussion on the case of  asymptotically power-law distributed steps is not exhaustive.
We refer to \cite{Levy_flights_Chechkin2008,MetzlerKlafter2000} (and the references therein) 
for a detailed analysis of the dynamics of $\alpha$-stable propagators.

The previous construction yields the following general expression of the NESS propagator. 

\begin{Proposition}
    For $d=1$  the NESS propagator \eqref{limiting} for $p\to 0^+$  takes the following form 
    \beq
\label{NESS}
P_c(y,\infty) = \int_0^{\infty}e^{-\tau} \tau^{-\frac{1}{\alpha}} {\cal L}_{\alpha}(y \tau^{-\frac{1}{\alpha}}){\rm d}\tau  , \hspace{1cm} y \in \mathbb{R},
\eeq
written in terms of ${\cal L}_{\alpha}(y)$ that denotes a specific  alpha-stable PDF which  depends on the random walk \eqref{random_position}.
\end{Proposition}

\begin{rmk}
   For $d=1$ ($A_1=A$, $y_1=y$) we identify (\ref{limit-large-lamda}) with (\ref{NESS}) where we consider 
${\cal L}_1(y) = \delta(y-A)$.

For $d=1$, (\ref{well_scales_propa}) takes the form (\ref{NESS}) where ${\cal L}_2(y)$  is Gaussian with mean zero and variance equal to $B$.

As mentioned above, for the Sibuya walk, \eqref{limit_biased_levy}
for $k'= k\tau^{\frac{1}{\alpha}}$ gives (\ref{NESS})  where $\ds {\cal L}_{\alpha}(y)= 
\frac{1}{2\pi}\int_{-\infty}^{\infty} e^{iky} e^{-(ik)^{\alpha}}{\rm d}k$ (index $\alpha \in (0,1)$) is a one-sided stable PDF.
\end{rmk}

\setcounter{equation}{0}
\renewcommand{\theequation}{A.\arabic{equation}}
%
%
%
%
%
%
%
%
\section{Monte Carlo simulation method}
\label{simulation_method}
%
%
%
%
Here we give a brief sketch of the numerical simulation method performed in this paper,  
such as Figs.~\ref{arrivals_def_ber}(b), and \ref{DSP_counting_var_simu}. To obtain a sample realization of  $N_{I}(t)$ or $N_{II}(t)$, we use the following general representation of the survival probability (see \cite{PachonPolitoRicciuti2021})
\beq
\label{survival_general_proba_res}
\Phi^{(0)}(t) = \prod_{k=1}^t (1-\alpha_k) ,\hspace{1cm} t \in \mathbb{N} , \hspace{1cm} \Phi^{(0)}(0) = 1,
\eeq
Representation (\ref{survival_general_proba_res}) is general and holds true for any transient or recurrent renewal process we consider. In this product, $\alpha_t \in [0,1]$ denote the conditional probabilities that an event occurs at time $t$, given no event has occurred up to and including time
$t-1$. The PDF that an event occurs at time $t$ is therefore $\psi(t)= \alpha_t\,\Phi^{(0)}(t-1)$
from which 
\beq
\label{alpha_t_explicit}
\alpha_t= \frac{\psi(t)}{\Phi^{(0)}(t-1)} ,\hspace{1cm} t \in \mathbb{N},
\eeq
with $\alpha_0=0$. 
Let $\mathcal{R} \in [0,1]$ be a random number which is independently generated at each integer time increment. 
A new event occurs at time $t$ (i.e., $t=J_{n+1}$ is a new arrival time) if $\mathcal{R} \leq \alpha_k$ ($k=t-J_n$ is the time delay 
from the previous latest arrival time $J_n<t$ to time $t$).
The method applies to either transient or recurrent renewal processes.
Clearly for the cases in which $ {\cal S}_{\infty} = \sum_{k=1}^{\infty} \alpha_k  < \infty$, one obtains $\Phi^{(0)}(\infty) ={\cal P} >0$ covering the class of transient renewal processes.
On the other hand, for divergent ${\cal S}_{\infty} = \sum_{k=1}^{\infty} \alpha_k = \infty$ the infinite product $\Phi^{(0)}(\infty)=0$ corresponds to recurrent renewal processes.
%
%

For instance, for the DBP we have (see (\ref{DFB-GF}), (\ref{DBP-survival}))
\beq
\label{DFB_alpha}
\alpha_t = \frac{\mathcal{Q}pq^{t-1}}{\mathcal{P}+\mathcal{Q}q^{t-1}} , \hspace{1cm} t \in \mathbb{N}.
\eeq
The RP is Markovian only if $\alpha_t$ is constant. Clearly this is the case only 
in the non-defective limit $\mathcal{Q}=1$ recovering the Bernoulli process for which (\ref{DFB_alpha}) gives $\alpha_k=p$.
The time dependence of (\ref{DFB_alpha}) shows non-markovianity of DBP and $\alpha_t$ decays geometrically in $t$, reflecting transience of the DBP with ${\cal S}_{\infty} < \infty$.
%

%
%
Finally, by a slight modification of the above procedure we introduced a stopping rule to obtain the simulations in Figs.~\ref{M-infinity} and \ref{BerstopsB}.
%

\end{appendix}

\end{document}